\documentclass[11pt]{amsart}
\usepackage[latin1]{inputenc}
\usepackage[english]{babel}
\usepackage[T1]{fontenc}
\usepackage{mathtools}
\usepackage{amssymb,amsmath,amsthm,amsfonts,enumerate}
\usepackage[autostyle]{csquotes}
\usepackage{alphalph,cite}
\usepackage{nameref}
\usepackage{verbatim}
\usepackage{mathrsfs} 
\usepackage{tikz-cd}

\makeatletter
\@namedef{subjclassname@2010}{ \textup{2010} Mathematics Subject Classification}
\makeatother
\theoremstyle{definition}
\newtheorem{theorem}{Theorem}[section]

\newtheorem{theoremletters}{Theorem}

\theoremstyle{definition}
\newtheorem{lemma}[theorem]{Lemma}
\theoremstyle{definition}
\newtheorem{corollary}[theorem]{Corollary}
\theoremstyle{definition}
\newtheorem{proposition}[theorem]{Proposition}
\theoremstyle{definition}

\theoremstyle{definition}
\newtheorem{remark}[theorem]{Remark}
\theoremstyle{definition}

\theoremstyle{definition}

\numberwithin{equation}{section}

\numberwithin{equation}{section}

\DeclareMathOperator{\supp}{supp}

\newcommand{\norm}[1]{\lVert#1\rVert}
\newcommand{\vertiii}[1]{{\left\vert\kern-0.25ex\left\vert\kern-0.25ex\left\vert #1 
    \right\vert\kern-0.25ex\right\vert\kern-0.25ex\right\vert}}
\newcommand{\abs}[1]{\left\lvert#1\right\rvert}

\newcommand{\distE}[1]{\emph{d}_X(#1)}
\newcommand{\distC}[1]{\emph{d}_{CX}(#1)}

\makeatletter
\def\author@andify{
	\nxandlist {\unskip ,\penalty-1 \space\ignorespaces}
	{\unskip {} \@@and~}
	{\unskip \penalty-2 \space \@@and~}
}
\makeatother

\begin{document}
\title{Quotients, $\ell_\infty$ and abstract Ces\`aro spaces}

\author{Tomasz Kiwerski}
\address[Tomasz Kiwerski]{Institute of Mathematics, Poznan University of Technology,
	Piotrowo 3A, 60-965 Pozna\'{n}, Poland}
\email{tomasz.kiwerski@gmail.com}

\author{Pawe\l {} Kolwicz}
\address[Pawe\l {} Kolwicz]{Institute of Mathematics, Poznan University of Technology,
	Piotrowo 3A, 60-965 Pozna\'{n}, Poland}
\email{pawel.kolwicz@put.poznan.pl}

\author{Jakub Tomaszewski}
\address[Jakub Tomaszewski]{Institute of Mathematics, Poznan University of Technology,
	Piotrowo 3A, 60-965 Pozna\'{n}, Poland}
\email{jakub.tomaszewski42@gmail.com}

\maketitle

\begin{abstract}
	Investigating some re-arrangement properties of the norm in the quotient spaces $X/X_a$ we determine
	the properties of the spaces $X$ guaranteeing the existence of a lattice isometric copy of $\ell_\infty$ in
	the abstract Ces\`aro spaces $CX$.
\end{abstract}

\bigskip

\begin{flushright}
	{\it Dedicated to the memory of} \\
	{\it Professor Henryk Hudzik} \\
	{\it (1945--2019)}
\end{flushright}

\bigskip

\renewcommand{\thefootnote}{\fnsymbol{footnote}} \footnotetext[0]{
	{\it Date:} \today.
	
	
	
	Research (all authors) supported by Ministry of Science and Higher Education of Poland, grant number 0213/SIGR/2154.}

	\section{Introduction and preliminaries}
	
	Let ${\mathscr P}$ be a property for which it makes sense to aks whether a Banach space has this property or not. For a Banach ideal space
	$X$ we can now pose the following problem: what condition, say $\varepsilon$, should we impose on the space $X$ so that the implication
	\begin{equation*} \tag{$\Delta$}
		X \in ({\mathscr P}) + \varepsilon \Longrightarrow CX \in ({\mathscr P})
	\end{equation*}
	holds true? Here $CX$ stands for the Ces\`aro space which can be viewed, for example, as an optimal domain for the Hardy operator
	$C \colon f \mapsto C(f)(x) \coloneqq \frac{1}{x}\int_0^x f(t) dt$,
	that is, the biggest in the sense of inclusion Banach ideal space such that the operator $C$ with fixed codomain space $X$ is still bounded.
	We should probably point out that there is really no reason (maybe, beyond some presonal interests) to restrict this question to
	the space $CX$, because the Ces\`aro spaces themselves constitute only a special, but important, subclass of a very broad family of
	{\it spaces generated via sublinear operators} (see \cite{Ast12}, \cite{CR06}, \cite{CR16}, \cite{DS07}, \cite{HW06a}, \cite{HW06b} and \cite{Mas91});
	this class of spaces includes, for example, K\"othe--Bochner spaces, real interpolation spaces, extrapolation spaces, Besov spaces, Triebel--Sobolev
	spaces and some other optimal domains associated with the kernel operators (like multiplication operator, Volterra operator, Ces\`aro operator,
	Copson operator, Poisson operator or Riemann--Liouville operator), differential operators, convolutions, the Fourier transform, the (finite)
	Hilbert transform and the Sobolev embedding (see \cite{ORS08} and references given there). Many, more or less classical results, can be seen
	through the prism ($\Delta$) of the above scheme, where the Ces\`aro operator $C$ is replaced by the appropriate sublinear operator, say, $S$;
	let us give few concrete examples of this type of results, which, however, remain close enough to the Ces\`aro spaces:
	(a) Astashkin \cite{Ast12} (solving the problem posed in \cite{Mas91}) proved that every non-trivial subspace of a Banach space $D_{p}(S)$
	generated by some positive sublinear operator $S$ and an $L_p$-space with $1 \leqslant p < \infty$ contains, for any $\varepsilon > 0$,
	an $(1+\varepsilon)$-copy of $\ell_p$ which is $(1 + \varepsilon)$-complemented in $D_p(S)$ (note that this result generalizes the well-known
	Levy result for the Lions--Petree interpolation spaces $(X_0,X_1)_{\theta,p}$ and, on the other hand, the earlier Astashkin and Maligranda
	result \cite{AM09} for the classical Ces\`aro function spaces $Ces_p$);
	(b) Hudzik and Wla\'zlak in \cite{HW06a} and \cite{HW06b} studied various convexity and monotonicity properties, say ${\mathscr G}$, of the
	space $D_E(S)$, where $E$ is a Banach ideal space. Essentially, they were able to show that if $E \in ({\mathscr G})$ and an injective positive
	sublinear operator $S$ has some special geometric properties corresponding to the property ${\mathscr G}$, then also $D_E(S) \in ({\mathscr G})$.
	This result implies immediate applications for the K\"othe--Bochner spaces and the Ces\`aro--Orlicz spaces, cf.~\cite{KK18a} for a more direct
	approach in the case of different types of the Ces\`aro function spaces;
	(c) Masty\l {}o in \cite{Mas91} and \cite{Mas92} investigated some structural properties of the space $D_E(S)$.
	
	In this paper we give a general and rather natural condition to ensure that the Ces\`aro space $CX$ contains a lattice isometric copy of $\ell_\infty$,
	as long as a rearrangement invariant space $X$ contains such a copy as well.
	
	A Banach space $(X, \norm{\cdot}_X)$ is said to be a {\it Banach lattice} if there is a lattice order on $X$, say $\leqslant$, such that
	for $x,y \in X$ with $\abs{x} \leqslant \abs{y}$ we have $\norm{x}_X \leqslant \norm{y}_X$. By $X_+$ we denote the {\it positive cone}
	of a Banach lattice $X$, that is, $X_+ = \{x \in X \colon x \geqslant 0 \}$.
	
	A mapping $T$ between two Banach lattices $X$ and $Y$ is said to be a {\it lattice} (or {an \it order}) {\it isomorphism} if it is a linear topological
	isomorphism which preserves the order $\leqslant$. If, additionally, the operator norm of a lattice isomorphism
	$T \colon X \rightarrow Y$ is equal one, then we will emphasize this fact by calling $T$ a {\it lattice} (or an {\it order}) {\it isometry}. 
	
	Let $(\Omega ,\Sigma ,\mu )$ be a complete $\sigma $-finite measure space. Denote by $L_{0}(\Omega) = L_{0}(\Omega,\Sigma,\mu)$ the set of all
	(equivalence classes of) real-valued $\mu$-measurable functions defined on a measure space $(\Omega,\Sigma,\mu)$.
	A \textit{Banach ideal space} $X=(X,\norm{\cdot}_X)$ on $(\Omega,\Sigma,\mu)$ is understood to be a Banach space $X$ such that $X$ is a linear
	subspace of $L_{0}(\Omega,\Sigma,\mu)$ satisfying the so-called \textit{ideal property}, which means that if $f,g\in L_{0}(\Omega,\Sigma,\mu)$,
	$\abs{f(t)} \leqslant \abs{g(t)}$ $\mu$-almost everywhere on $\Omega$ and $g\in X$, then $f\in X$ and $\lVert f\rVert_{X}\leqslant \lVert g\rVert _{X}$.
	If it is not stated
	otherwise we assume that a Banach ideal space $X$ on $(\Omega,\Sigma,\mu)$ contains a function which is positive $\mu$-almost everywhere on
	$\Omega$ (such a function is called the \textit{weak unit} in $X$). For a function $f\in L_{0}(\Omega,\Sigma,\mu)$ we define a \textit{support of}
	$f$ as a set $\supp(f)\coloneqq \{x\in \Omega \colon f(x) \neq 0\}$. Moreover, by a \textit{support} $\supp(X)$ of the Banach ideal space
	$X$ on $(\Omega,\Sigma,\mu)$ we mean the smallest (in the sense of inclusion) $\mu$-measurable subset, say $A$, of $\Omega$ such that
	$f\chi_{\Omega \setminus A} = 0$ for all $f\in X$. The existence of the weak unit is equivalent to the following condition that $\supp(X)=\Omega$.
	We say that a Banach ideal space $X$ is \textit{non-trivial} if $X \neq \left\{ 0 \right\}$.
	
	For two Banach ideal spaces $X$ and $Y$ on $\Omega$ the symbol $X \overset{E}{\hookrightarrow} Y$ means that the embedding $X \subset Y$ is
	continuous with the norm not bigger than $E > 0$, that is to say, $\norm{f}_Y \leqslant E\norm{f}_X$ for all $f \in X$. If the embedding
	$X \overset{E}{\hookrightarrow} Y$ holds with some (maybe unknown) constant $E > 0$ we simply write $X \hookrightarrow Y$. Moreover, the symbol
	$X = Y$ ($X \equiv Y$) means that the spaces are the same as a sets and the norms are equivalent (equal, respectively).
	
	We say that an element $x$ in a Banach lattice $X$ is {\it order continuous} (or has {\it order continuous norm})
	if for any sequence $\{x_n\}_{n=1}^\infty$ in $X$ satisfying $0 \leqslant x_n \leqslant \abs{x}$ and $x_n \downarrow 0$ (that is,
	$x_{n+1} \leqslant x_n$ and $\inf_{n \in \mathbb{N}}\{x_n\} = 0$)
	we have $\norm{x_n} \rightarrow 0$ as $n \rightarrow \infty$. By $X_a$ we denote the subspace of all order continuous elements
	of $X$. If $X = X_a$, that is, every element of $X$ is order continuous, then the space $X$ is said to be {\it order continuous}.
		
	By $X^*$ we denote the {\it topological dual space} of a Banach lattice $X$. If $X^{dd} = X$ then we have the {\it Yosida--Hewitt decomposition}
	of the space $X^*$, namely
	\begin{equation*}
		X^* = X_a^* \oplus X^s,
	\end{equation*}
	where $X_a^*$ is the space of all {\it order continuous functionals}, that is to say, $x^* \in X_a^*$ if and only if $x^*(x_n) \rightarrow 0$
	as $n \rightarrow \infty$ for any sequnce $\{x_n\}_{n=1}^\infty \subset X_+$ and $x_n \downarrow 0$,
	and $X^s$ is the space of all {\it singular functionals}, that is to say, $x^* \in X^s$ if and only if $x^*(x) = 0$ for any $x \in X_a$.
	For an order continuous Banach lattice $X$ the spaces $X_a^*$ and $X^*$ coincide. If $X$ is a Banach ideal space on $(\Omega,\Sigma,\mu)$
	such that $\supp(X_a) = \supp(X)$ then the space $X_a^*$ coincide with the {\it K\"{o}the dual} space (or {\it associated space}) $X'$
	\begin{equation*}
		X' \coloneqq \left\{f \in L_0(\Omega) \colon \sup\limits_{\norm{g}_X \leqslant 1} \int_\Omega \abs{f(t)g(t)}d\mu < \infty \right\}.
	\end{equation*}
	Recall that $X \overset{1}{\hookrightarrow} X''$ and we have the equality $X \equiv X''$ if and only if the norm in $X$ has the {\it Fatou property},
	that is to say, if the conditions $0 \leqslant f_{n}\uparrow f \in L_{0}$ with $\{f_{n}\}_{n=1}^\infty \subset X$ and $\sup_{n\in \mathbb{N}}\Vert f_{n}\Vert_{X}<\infty$
	imply that $f\in X$ and $\Vert f_{n}\Vert_{X}\uparrow \Vert f\Vert _{X}$.
	
	A closed linear subspace $J$ of a Banach lattice $X$ is called an {\it order ideal} (or simply {\it ideal}) if it has
	the so-called {\it ideal property} which means that if $f \in J$, $g \in X$ and $\abs{g} \leqslant \abs{f}$ then also $g \in J$
	(subspaces with this property are also called {\it solid}). It is clear that $\{ 0 \}$ and the
	space $X$ itself are order ideals of $X$. Non-trivial example of an order ideal is the closure (in norm topology of the Banach ideal space $X$)
	of the set of simple functions $X_b$ (see \cite[Theorem~3.11, p.~18]{BS88}) and $X_a$, that is, the subspace of all functions
	with order continuous norm in $X$ (see \cite[Theorem~3.8, p.~16]{BS88}). An intersection and a sum of any two ideals is again an ideal
	(cf. \cite[Proposition~1.2.2, p.~12]{MN91}). If $J$ is a closed ideal in a Banach lattice $X$, then $X / J$ is a Banach lattice
	with respect to the quotient norm, that is to say, $\norm{x}_{X / J} = \inf\{\norm{y}_X \colon Q(y) = Q(x) \}$, where
	$Q \colon J \rightarrow X / J$ is the cannonical quotient map (cf. \cite[Corollary~1.3.14, p.~32]{MN91}).
	
	In this paper we will deal with Banach lattices but the main focus will be on {\it rearrangement invariant spaces} (or {\it symmetric spaces}),
	that is to say, the Banach ideal spaces with the additional property that for any two equimeasurable functions $f, g \in L_{0}$
	(which means that they have the same distribution functions $d_{f} \equiv d_{g}$, where
	$d_{f}(\lambda) \coloneqq \mu(\{t\in \Omega \colon |f(t)| > \lambda\})$ for $\lambda \geqslant 0$)
	and if $f\in X$ then $g \in X $ and $\Vert f\Vert _{X}=\Vert g\Vert _{X}$. Due to the Luxemburg representation theorem, it is enough to
	consider rearrangement invariant spaces only on three separable measure spaces, namely, the unit interval $\left( 0,1\right)$ or the semi-axis
	$\left( 0,\infty \right)$ with the Lebesgue measure $m$ (in this case we will talk about {\it function spaces}) and the set of positive
	integers $\mathbb{N}$ with the counting measure $\#$ (and then we will refer to {\it sequence spaces}).
	Moreover, by a {\it non-increasing rearrangement} of a function $f \colon \Omega \rightarrow \mathbb{R}$ we mean the function
	$f^* \colon [0,\infty) \rightarrow [0,\infty]$ defined as
	\begin{equation*}
		f^{\ast }(t) = \mathrm{inf}\{\lambda > 0 \colon d_{f}(\lambda) \leqslant t\} \text{ for } t \geqslant 0,
	\end{equation*}
	under the convention $\mathrm{inf}\{\emptyset\} = \infty$. In the sequence case, however, we need a little modification of the above definition,
	namely $x_n^* = \mathrm{inf}\{\lambda > 0 \colon d_x(\lambda) < n\}$, where $x = \{x_n\}_{n=1}^\infty$.
	
	The \textit{fundamental function} $\varphi _{X}$ of a rearrangement invariant space $X$ on $\Omega$ is defined by the following formula
	\begin{equation*}
		\varphi _{X}(t)=\Vert \chi _{(0,t)}\Vert _{X} \text{ for } t > 0,
	\end{equation*}
	(with a fairly obvious modification when $\Omega = \mathbb{N}$, i.e., $\varphi_X(n) = \norm{\sum_{k=1}^n e_k}_X$ for $n \in \mathbb{N}$,
	where $\{e_n\}_{n=1}^\infty$ is the canonical basic sequence of $X$) where $\chi_A$, throughout, will denote the characteristic function
	of a set $A$. It is well-known that fundamental function is quasi-concave on $\Omega$, that is: $\varphi _{X}(0) = 0$; $\varphi _{X}$
	is positive and non-decreasing; $t \mapsto \varphi _{X}(t)/t$ is non-increasing for $t > 0$ or, equivalently,
	$\varphi _{X}(t) \leqslant \max\{1,t/s\}\varphi _{X}(s)$ for all $s,t \in \Omega$.
	
	For some general properties of Banach ideal and rearrangement invariant spaces we refer, for example, to \cite{BS88}, \cite{KPS82},
	\cite{LT77}, \cite{LT79} and \cite{Ma89}. More information about Banach lattices can be found, for example, in \cite{AB85}, \cite{KA82} and \cite{MN91}
	
	Let us recall some examples of rearrangement invariant spaces. Each increasing concave function $\varphi$ on $\Omega$
	generates the {\it Lorentz function space $\Lambda_\varphi$} on $\Omega$ endowed with the norm
	\begin{equation*}
		\norm{f}_{\Lambda_\varphi} = \int_\Omega f^*(t)d\varphi(t) = \varphi(0^+)\norm{f}_{\infty} + \int_0^{m(\Omega)} f^*(t)\varphi'(t)dt < \infty.
	\end{equation*}
	
	Recall also that for a {\it quasi-concave function $\varphi$} the \textit{Marcinkiewicz function space} $M_{\varphi }$ on $\Omega$ is defined in
	the following way
	\begin{equation*}
		M_{\varphi} = \left\{ f\in L_{0}(\Omega) \colon \norm{f}_{M_{\varphi}} = \sup_{t \in \Omega}\frac{\varphi(t)}{t}\int_0^t f^*(s)ds < \infty \right\}.
	\end{equation*}
	
	Lorentz and Marcinkiewicz spaces play quite a special role among rearrangement invariant spaces, namely they are the smallest and, respectively,
	the largest (in a sense of inclusion) rearrangement invariant spaces with a given fundamental function. More precisely, for a given rearrangement
	invariant space $X$ with the fundamental function $\varphi$ (note that every such a function is equivalent to a concave function) we have the
	following embeddings
	\begin{equation*}
		\Lambda_\varphi \overset{2}{\hookrightarrow} X \overset{1}{\hookrightarrow} M_\varphi.
	\end{equation*}
	
	Let $F$ be an non-decreasing convex function on $[0,\infty)$ such that $F(0) = 0$ and let $X$ be a Banach ideal space on $\Omega$.
	The {\it Calder\'on--Lozanovski{\u \i} space $X_F$} on $\Omega$ is defined as the space of all measurable functions $f \colon \Omega \rightarrow \mathbb{R}$
	for which the followng norm, the so-called {\it Luxemburg--Nakano norm}, is finite
	\begin{equation*}
		\norm{f}_{X_F} = \inf\left\{\lambda > 0 \colon \norm{F(f/\lambda)}_X \leqslant 1 \right\} < \infty.
	\end{equation*}
	The spaces $X_F$ were introduced by Calder\'on \cite[p.~122]{Cal64} and Lozanovski{\u \i} \cite{Lo65}.
	This is a Banach ideal space. Moreover, if $X$ is a rearrangement invariant space with the Fatou property, then the space $X_F$ is like that as well.
	Note also that in the case when $X = L_1$ the space $X_F$ is just the {\it Orlicz space $L_F$} equipped with the Luxemburg--Nakano norm.
	On the other hand, if $X$ is a Lorentz space $\Lambda_\varphi$, then $X_F$ is the corresponding {\it Orlicz--Lorentz space $\Lambda_{F,\varphi}$}.
	It is also clear that in the special situation, when $F(t) = t$, the Orlicz--Lorentz space $\Lambda_{F,\varphi}$ coincide, up to equality of norms,
	with the Lorentz space $\Lambda_\varphi$.
	Finally, if $F(t) = t^{p}$, where $1 \leqslant p < \infty$, then the space $X_F$ is the {\it $p$-convexification $X^{(p)}$} of the space $X$
	equipped with the norm $\norm{f}_{X^{(p)}} = \norm{\abs{f}^p}_X^{1/p}$ and if $F(t) = 0$ for $0 \leqslant t \leqslant 1$ and $F(t) = \infty$
	for $t > 1$, then $X_F \equiv L_\infty$.
	
	In the case of sequence spaces we prefer a slight modification of the introduced notation, that is, if this make sense, we will use small letters
	to denote the space, for example, we will just denote the Lorentz space $\Lambda_\varphi$ on $\mathbb{N}$ as $\lambda_\varphi$ and the Orlicz space
	$L_F$ on $\mathbb{N}$ as $\ell_F$, etc.
	
	We will use the notation $A \preccurlyeq B$ or $B \succcurlyeq A$ to denote an estimate of the form $A \leqslant CB$ for some constant $C > 0$
	depending on the involved parameters only. We also write $A \asymp B$ for $A \preccurlyeq B \preccurlyeq  A$. Other notations and definitions
	will be introduced as needed.
	
	Now we give a brief overview of the paper.
	
	Section 2 is slightly different in nature to the rest of this article. Here we revisit some Hudzik's \cite{Hu98} characterizations of Banach
	lattices containing a lattice isometric copies of $\ell_\infty$ and complete one of them (Theorem~\ref{Hudzik w druga strone}). This connection
	between, on the one hand, the existence of a lattice isometric copy of $\ell_\infty$ in a Banach lattice $X$ and, on the other, with the fact
	that the unit sphere in the space $X$ must contain an element, say $x$, with the property that the distance from $x$ to the order ideal of all
	order continuous elements in $X$ is equal exactly one, will be our starting point and the leitmotif to which we will come back in later parts
	of our work. 
	
	Section 3 is rather technical in character. In short, we show there that the quotient spaces $X/X_a$, where $X_a$ is the ideal of all order
	continuous elements in a rearrangement invariant space $X$, behave somewhat similar to a rearrangement invariant space, that is to say,
	they have the {\it ideal property} (Corollary~\ref{colnierownosc}) and they are {\it symmetric} (Theorem~\ref{fgwiazdka}). We will
	also explain how our formulas are related to some of de Jonge's earlier results from \cite{deJ77}.
	
	In Section 4 we will focus on Ces\`aro spaces. Using the results from the previous sections we prove some general theorems about a lattice
	isometric copies of $\ell_\infty$ in these spaces (Theorem~\ref{CX ma izometryczne l_infty} and Theorem~\ref{CX ma kopie linfty Xa = 0})
	and justify (which is not so easy) that they really generalize the so far known results obtained in the class of Ces\`aro--Orlicz spaces
	in \cite{KK09} and \cite{KK18b}.
	
	Finally, at the end of this paper, we include Appendix, which was created for the purposes of \S4 (and, at the same time, it complements the
	results obtained by the first and last author in \cite{KT17}). It contains a fairly satisfactory description of the ideal of all order
	continuous functions in the Ces\`aro sequence spaces $CX$ expressed in the language of the ideal $X_a$.
	
	\section{About isometric copies of $\ell_{\infty}$ one more time}
	
	Lozanowski{\u \i}'s well-known result from \cite{Lo69} gives some geometric characterization of the order continuity property.
	More precisely, it states that a Banach lattice $X$ contains a lattice isomorphic copy of $\ell_\infty$ if and only
	if the space $X$ is not order continuous. Of course, every isomorphic copy of $\ell_\infty$ is automatically complemented in
	the ambient space due to the fact that the space $\ell_\infty$ is isometrically injective. Note also that a Banach lattice $X$ contains a
	(lattice) almost isometric copy of $\ell_\infty$ whenever it contains a (lattice) isomorphic copy of $\ell_\infty$ (see \cite[Theorem~1]{HM93}
	and \cite[Theorem~3]{Pa81}). It seems to be worth comparing this statement with the classical James result \cite{Ja64} (often called simply
	the {\it James's distortion theorem}), which describes a similar phenomenon but for spaces $c_0$ and $\ell_1$. Moreover, if a Banach space
	$X$ contains an asymptotically isometric copy of $\ell_\infty$, then it contains even an isometric copy of $\ell_\infty$ (see \cite[Theorem~6]{Do00} and
	\cite[Theorem~2.5]{CHL17} for a \enquote{lattice version}).
	However, in general, the lack of order continuity does not guarantee the existence of {\it any} isometric copy of $\ell_\infty$.
	Indeed, let $X$ and $Y$ be two Banach function spaces on $\Omega$ with non-trivial intersection $X \cap Y$ such that the first
	one is not order continuous and the second one is strictly convex (for example, we can take $X = L_\infty$ and $Y = L_2$).
	Then the space
	\begin{equation*}
		\mathcal{X} = X \cap Y \text{ with the norm } \norm{f}_{\mathcal{X}} = \norm{f}_X + \norm{f}_Y,
	\end{equation*}
	is strictly convex and is not order continuous. Consequently, in view of the Lozanowski{\u \i} result \cite{Lo69}, it contains
	a lattice isomorphic copy of $\ell_\infty$ but clearly cannot contain an isometric copy of $\ell_\infty$.
	Note also that the same {\it trick} will work if instead of strict convexity we will require that the second space is strictly
	monotone. In this case, however, the space $Z$ will be strictly monotone and not order continuous, so it will contains a
	lattice isomorphic but not lattice isometric copy of $\ell_\infty$. Let us now recall the following classical result.
	
	\begin{theoremletters}[Riesz's lemma]
		{\it If $Y$ is a proper closed linear subspace of a normed space $X = (X, \norm{\cdot})$, then for any $0 < \varepsilon < 1$ there
			exists $x \in S(X)$ such that} $\distE{x, Y} \geqslant 1 - \varepsilon$.
	\end{theoremletters}
	
	It can be demonstrated that the above result is in general not true for $\varepsilon = 0$. Actually, for a Banach space $X$ to have the property
	that given a proper closed linear subspace $Y$ of $X$ there exists $x \in S(X)$ with $\distE{x, Y} = 1$ it is necessary and sufficient
	that the space $X$ is reflexive (cf. \cite[Notes and Remarks on p.~6]{Di84}). In fact, James \cite{Ja57} proved that a Banach space $X$
	is reflexive if and only if every continuous linear functional on $X$ is norm attaining. Therefore, if $X$ is nonreflexive, then there is a linear
	functional, say $\varphi$, in $S(X^*)$ which does not achieve its norm. Taking $Y = \text{ker}(\varphi)$ we see immediately that $Y$ is a proper
	closed linear subspace of $X$ and there is no such element $x \in S(X)$ with $\distE{x, Y} = 1$. On the other hand, if $X$ is reflexive
	and $Y$ is a closed linear subspace of $X$, then applying the Hahn--Banach theorem and bearing in mind the result of James we see that there
	exists $\psi \in S(X^*)$ such that $Y \subset \text{ker}(\psi)$ and $x \in S(X)$ with $\psi(x) = 1$. Consequently, we have the following inequalities:
	\begin{equation*}
		\norm{x - y} \geqslant \abs{\psi(x) - \psi(y)} = \abs{\psi(x)} = 1 \text{ for all } y \in Y,
	\end{equation*}
	that is, $\distE{x, Y} = 1$. On the other hand, the existence of an element realizing the distance from the ideal $X_a$, where now $X$
	is a Banach ideal space, is closely related to the fact that the space $X$ contains a lattice isometric copy of $\ell_\infty$.
	Anyhow, for $X = \ell_\infty$ and $f = \chi_{\mathbb{N}}$ we have
	$\norm{f}_{X} = \distE{\chi_{\mathbb{N}}, X_a = c_0} = 1$ and, in a sense, this is the generic case.

	\begin{theoremletters}[H. Hudzik, 1998 \cite{Hu98}] \label{Twierdzenie Hudzika}
		{\it Let $X$ be a super $\sigma$-Dedekind complete Banach lattice with the semi-Fatou property
			such that $(X_a^d)^d = X$ \footnote{In less general situation, e.g. when $X$ is a Banach ideal space, this condition simply
			means that $\supp(X_a) = \supp(X)$ (cf.~\cite[p.~523]{Hu98}).}. If we can find an element $f \in X$ such that} $\norm{f}_X = \distE{f, X_a} = 1$,
			{\it  then $X$ contains a lattice isometric copy of $\ell_\infty$.}
	\end{theoremletters}

	Let's us consider the space $\mathfrak{X} = (L_\infty, \norm{\cdot}_{\mathfrak{X}})$, where $\norm{\cdot}_{\mathfrak{X}}$ is an equivalent norm on $L_\infty$
	given by the formula: $\norm{f}_{\mathfrak{X}} = \norm{f}_{L_\infty} + \norm{f}_{Y}$ for $f \in L_\infty$, and $Y$ is strictly convex
	space such that $L_\infty \hookrightarrow Y$. Then the subspace $\mathfrak{X}_a$ is trivial and
	$$\distE{f, \mathfrak{X}_a} = \distE{f, \{ 0 \}} = \norm{f}_{\mathfrak{X}} = 1 \text{ for all } f \in S(\mathfrak{X}).$$
	However, the space $\mathfrak{X}$ cannot contain a lattice isometric copy of $\ell_\infty$ as we have already explained above.
	Therefore the assumption $(X_a^d)^d = X$ in Theorem \ref{Twierdzenie Hudzika} about the support of the ideal $X_a$ cannot be omitted
	in general.
	
	Needless to say, the question whether the condition from Theorem \ref{Twierdzenie Hudzika} is also sufficient imposes itself.
	We believe that, this fact belongs to {\it folklore} but, as far as we know, it was never explicitly mentioned (even in \cite{Hu98}).
	To fill this gap and avoid the impression (which the authors themselves had) that this condition may only be necessary, we will give
	a short proof which completes Hudzik's result.
	
	\begin{theorem} \label{Hudzik w druga strone}
		{\it Assume that $X$ is a Banach ideal space over a $\sigma$-finite measure space with the semi-Fatou property
			such that $\supp(X_a) = \supp(X)$. The space $X$ contains a lattice isometric copy of $\ell_\infty$
			if and only if there exists an element $f \in X$ with} $\norm{f}_X = \distE{f, X_a} = 1$.
	\end{theorem}
	\begin{proof}
		Using Theorem 1 from \cite{Hu98} we can find a sequence $\{ f_n \}_{n=1}^\infty \subset S(X)$ with
		$\supp(f_n) \cap \supp(f_m) = \emptyset$ for $n \neq m$ such that $\norm{\sum_{n=1}^{\infty} f_n}_X = 1$.
		Take $g \in X_a$ and observe that
		\begin{align*}
		\norm{\sum_{n=1}^{\infty} f_n - g}_X & \geqslant \sup\limits_{n \in \mathbb{N}} \norm{f_n - g\chi_{\supp(f_n)}}_X \\
								& \geqslant \abs{\norm{f_n}_X - \norm{g\chi_{\supp(f_n)}}_X} \rightarrow 1,
		\end{align*}
		because  $\norm{g\chi_{\supp(f_n)}}_X \rightarrow 0$ as $n \rightarrow \infty$.
	\end{proof}

	Reffering to Theorem \ref{Hudzik w druga strone}, if a Banach ideal space $X$ contains a lattice isometric copy
	of $\ell_\infty$, then we can find  $f \in X$ with $\norm{f}_X = \distE{f, X_a} = 1$. So, looking from
	a slightly different perspective, this means that $\norm{f}_{X / X_a} = 1$ and there exist a singular functional
	$S \in S(X^*)$ with $S(f) = 1$ (note only that $(X / X_a)^* \approx X_a^{\perp} = X^s$, where $X^s$ is the space of
	all singular functionals, that is to say, those $S \in X^*$ for which $S(x) = 0$ for all $x \in X_a$).
	The converse of this statement is also true but before we give a short proof let us note the following result.
	
	\begin{lemma} \label{Lemat z funkcjonalem singularnym}
		{\it Let $X$ be a Banach lattice. If $S$ is a singular functional on $X$ then we have the following formula}
			\begin{equation*} \label{norma f. singularnego}
			\norm{S}_{X^{*}} = \sup\left\{ \frac{S(f)}{\distE{f,X_a}} \colon f \in X \setminus X_a \text{ with } \norm{f}_X = 1 \right\}.
			\end{equation*}
	\end{lemma}
	\begin{proof}
		First, observe that $S(f) = 0$ for all $f \in X_a$, so
		\begin{align*}
		\norm{S}_{X^{*}} & = \sup\left\{ S(f) \colon f \in X \text{ with } \norm{f}_X = 1 \right\} \\
		& = \sup\left\{ S(f) \colon f \in X \setminus X_a \text{ with } \norm{f}_X = 1 \right\} \\
		& \leqslant \sup\left\{ \frac{S(f)}{\distE{f,X_a}} \colon f \in X \setminus X_a \text{ with } \norm{f}_X = 1 \right\},
		\end{align*}
		where the last inequality follows from the fact that if $\norm{f}_X = 1$, then $\distE{f, X_a} \leqslant 1$.
		Therefore, we only need to prove the reverse inequality. To do this, take $f \in X \setminus X_a$ and let
		$\{ f_n \}_{n=1}^\infty \subset X_a$ be a sequence that realizes the distance of the function $f$ from the ideal $X_a$,
		that is to say, $\lim_{n\rightarrow\infty} \norm{f - f_n}_X = \distE{f, X_a}$. Now, since $S(f) = S(f - f_n)$, so
		\begin{equation*}
		S(f) \leqslant \norm{S}_{X^{*}} \norm{f - f_n}_X \rightarrow \norm{S}_{X^{*}} \distE{f, X_a} \text{ as } n \rightarrow \infty.
		\end{equation*}
		Thus, we obtain that
		\begin{equation*}
		\norm{S}_{X^{*}} \geqslant \frac{S(f)}{\distE{f,X_a}},
		\end{equation*}
		which ends the proof.
	\end{proof}

	Note that the above lemma for Orlicz--Lorentz spaces was proved in \cite{KLT19} basing on the modular space structure.
	
	Consequently, if there exists a norm attaining singular functional on $X$, say $S$, then $S(f) = \norm{S}_{X^*}$ for some $f \in B(X)$,
	i.e., $\distE{f, X_a} = 1$ (cf. Lemma~\ref{Lemat z funkcjonalem singularnym}) and the space $X$ contains a lattice 
	isometric copy of $\ell_\infty$ in view of Theorem \ref{Hudzik w druga strone}. What we have said so far may be summarized as follows.
	
	\begin{proposition} \label{Funkcjonal singularny a kopia}
		{\it Let $X$ be a super $\sigma$-Dedekind complete Banach lattice with the semi-Fatou property such that $(X_a^d)^d = X$.
			 Then $X$ contains a lattice isometric copy of $\ell_\infty$ if and only if there exists a singular functional
			 on $X$ which attains its norm.}
	\end{proposition}	
	
	We will now (seemingly) deviate for a moment from the course choosen so far. In \cite{GH96} (see also \cite{Wn84}; cf.~\cite{SS83}) Granero and Hudzik extended
	and completed the results obtained in \cite{LW83} for the space $\ell_\infty/c_0$\footnote{As mentioned in \cite{LW83}: {\it The Banach space $\ell_\infty/c_0$
			certainly falls into the category of a \enquote{classical Banach space}. Not only has it been around since the time of Banach's original
			monograph (1932), but it is also classical in the sense of Lacey or Lindenstauss and Tzafriri since it is isomorphic to the space
			$C(\beta\mathbb{N}\setminus\mathbb{N})$} (just to be clear $\beta\mathbb{N}$, as usual, is the {\v C}ech--Stone compactification of $\mathbb{N}$ endowed
			with the discrete topology, that is, $\beta\mathbb{N}$ is the unique compact Hausdorff space containing $\mathbb{N}$ as a dense subspace
			so that every bounded continuous function on $\mathbb{N}$ extends to a continuous function on $\beta\mathbb{N}$); we also recommend taking a look at\cite{DR91}.}
	considering a more general construction $\ell_{F}/(\ell_{F})_a$, where $\ell_{F}$ is the Orlicz sequence space (in fact, they even
	allowed for $\ell_{F}$ to be a Fr{\'e}chet space). Of course, we can go a step futher and consider a {\it neo-classical} Banach
	space $X/X_a$, where $X$ is a Banach lattice. Despite this somewhat baroque name, it is clear that $X/X_a$ is sometimes {\it classical},
	for example, if $X$ is a rearrangement invariant Banach sequence space with the Fatou property such that $\varphi_X(\infty) < \infty$,
	then $X / X_a = \ell_\infty / c_0 \approx C(\beta\mathbb{N} \setminus \mathbb{N})$.
	
	\begin{proposition}
		{\it Let $X$ be a Banach sequence space such that $X \hookrightarrow \ell_\infty$, the ideal $X_a$ is non-trivial and $X_a \neq X$.
			Then $X/X_a$ is not a dual space.}
	\end{proposition}
	\begin{proof}
		It is well-known that the space $\ell_\infty/c_0$ contains an isomorphic copy of $c_0(\Gamma)$, where $\text{card}(\Gamma) = \mathfrak{c}$
		and $\mathfrak{c}$ denotes the cardinal of the continuum, that is, $\mathfrak{c} = 2^{\aleph_0}$.
		Let us briefly remind that we can proceed as follows: (1) for each irrational number $r$, take a sequence $\{ q_n^r \}_{n=1}^\infty$
		of distinct rational numbers converging to $r$ and put $F_r = \{ q_n^r \colon n \in \mathbb{N}\}$; (2) the sets $\{ F_r \}_{r \in \mathbb{R}}$
		form an uncountable family of infinite sets such that $F_r \cap F_{r'}$ is finite whenever $r \neq r'$; (3) observe that $\ell_\infty \equiv \ell_\infty(\mathbb{Q})$;
		(4) the closed linear span of $\{ \pi(\chi_{F_r}) \colon r \in \mathbb{R} \}$, where $\pi \colon \ell_\infty \rightarrow \ell_\infty/c_0$ is
		the cannonical quotient map, is isometric to $c_0(\mathbb{R})$ (cf. footnote on page 19 in \cite{Ro70}).
		Now, if $X/X_a$ were a dual space, then thanks to Rosenthal's generalization of the classical Bessaga and Pe{\l}czy{\'{n}}ski result
		\cite[Corollary 1.5]{Ro70}, the space $X/X_a$ should have a copy of $\ell_\infty(\Gamma)$. But this is impossible, because
		\begin{equation*}
		\text{card}(\ell_\infty(\Gamma)) = 2^\mathfrak{c} > \mathfrak{c} = \text{card}(\ell_\infty) \geqslant \text{card}(X) \geqslant \text{card}(X/X_a),
		\end{equation*}
		where the second inequality follows from the fact that $X \hookrightarrow \ell_\infty$.
	\end{proof}

	The above result is a direct generalization of \cite[Proposition~6.2]{GH96} which was proved in the case when $X = \ell_F$.

	Let us now recall the well-known Phillips--Sobczyk theorem, which states that $c_0$ is not complemented in $\ell_\infty$ (to be precise,
	what Phillips proved was that $c$ is not complemented in $\ell_\infty$; look at \cite{CSCY00}, cf.~\cite[pp.~44--48]{AK06}). We will close this
	paragraph by re-proving a \enquote{lattice} variant of this theorem and exploring the Bourgain's approach \cite{Bo80} to this phenomenon.
	
	\begin{proposition}[G. Ja. Lozanovski{\u \i}, 1973]
		{\it Let $X$ be a Banach ideal space on a complete $\sigma$-finite measure space $(\Omega, \Sigma, \mu)$ such that $X_a \neq X$.
			Then the subspace $X_a$ is not complemented in $X$ except for the case when $X_a$ is trivial.}
	\end{proposition}
	\begin{proof}
		Suppose that $X$ is as above and consider the following diagram
		\begin{equation*}
		c_0 \hookrightarrow \ell_\infty \overset{T}{\longrightarrow} X \overset{Q}{\longrightarrow} X / X_a,
		\end{equation*}
		where $T$ is a lattice isomorphism and $Q$ is the canonical quotient map. It follows that $T(c_0) \subset X_a$
		(see Proposition~1 and Theorem~1 in \cite{Wo05}; cf. \cite{PW07}). Thus, if $q \colon \ell_\infty \rightarrow \ell_\infty / c_0$
		is a quotient mapping, then there exist a unique mapping $U \colon \ell_\infty / c_0 \rightarrow X / X_a$ defined in the following
		way $U \colon x + c_0 \mapsto T(x) + X_a$. It is not hard to see that $QT = Uq$. In short, the following diagram commutes
		\begin{equation*}
			\begin{tikzcd}
			\ell_\infty \arrow[r, "T"]
			\arrow[d, "q"]
			& X \arrow[d, "\text{Q}"] \\
			\ell_\infty / c_0 \arrow[r, "U"]
			& X / X_a
			\end{tikzcd}
		\end{equation*}
		and the space $X / X_a$ contains a lattice isomorphic copy of $\ell_\infty / c_0$. Now, observe that $X$ can be renormed to be a strictly
		convex space. In fact, let us consider the multiplication operator $M_y$ given by $M_y \colon X \ni x \mapsto xy \in L_1(\mu)$,
		where $y \in X'$ is chosen so that $\norm{y}_{X'} = 1$ and $y(t) > 0$ for $\mu$-a.e. $t \in \Omega$. Then we can define the functional
		$\sharp \cdot \sharp$ on $X$ as follows
		\begin{equation*}
		\sharp x \sharp = \norm{x}_X + \norm{M_y(x)}_{L_1(\mu)},
		\end{equation*}
		Now, since the space $\ell_\infty$ is universal for all separable Banach spaces (see \cite[Theorem~2.5.7, p.~46]{AK06})
		and admit strictly convex renorming (see \cite[p.~94]{Di76}), so without losing generality we can assume that $L_1(\mu)$
		is already strictly convex; another words, we can always equip the space $L_1(\mu)$ with a new norm, say $\abs{\cdot}$,
		such that the space $(L_1(\mu), \abs{\cdot})$ is strictly convex. But this gives that
		$\sharp \cdot \sharp$ is an equivalent norm on $X$ and that the space $(X, \sharp \cdot \sharp)$ is strictly convex as well (this follows
		from the well-known fact that if $T \colon X \rightarrow Y$ is an injective operator and $Y$ is strictly convex Banach space then $X$
		can be renormed to be also strictly convex; see \cite[p. 94]{Di76} for more details).
		On the other hand, Bourgain \cite{Bo80} proved that the space $\ell_\infty/c_0$ does not admit an equivalent strictly
		convex norm while, as we showed above, $X$ does. Consequently, the space $X/X_a$, which contains an isomorphic copy of $\ell_\infty / c_0$,
		is not isomorphic to a subspace of $X$ and this implies that $X_a$ is not complemented in $X$.
	\end{proof}

	Another proof of this fact can be found, for example, in \cite[Theorem~1.21, pp.~62--63]{Wn99}.
	
	Note also that Partington \cite[Theorem~1]{Pa81} proved that the space $X / X_a$ not only cannot be renormed to be strictly convex, but even
	contains a lattice isometric copy of $\ell_\infty$ (cf. \cite[Proposition~5.1]{GH96} and \cite[Theorem~2]{Wo05}).

\section{Some properties of the norm in the quotient space}

	It is clear that the quotient space $X/M$, where $X$ is a normed space and $M$ is a closed subspace of $X$, is also a normed space
	with respect to the norm $\norm{\cdot}_{X/M}$ defined as
	\begin{equation*}
	\norm{x + M}_{X/M} \coloneqq \inf\{ \norm{y}_X \colon y \in x + M \}.
	\end{equation*}
	In simple words, this norm measures the distance from a coset to the origin of the space $X/M$, that is
	\begin{equation*}
	\norm{x + M}_{X/M} = \distE{x + M, 0 + M} = \distE{x, M} \coloneqq \inf\{ \norm{x - m}_X \colon m \in M \}.
	\end{equation*}
	Therefore, the study of the function $x \mapsto \distE{x, M}$ is parallel to the study of certain properties of the quotient space $X/M$
	and vice versa. For example, de Jonge \cite{deJ77} proved that a Banach lattice $X$ is a {\it semi-$M$ space} (without going into detail,
	this class includes Orlicz spaces $L_F$ with both the Orlicz and the Luxemburg--Nakano norm and Lorentz spaces $\Lambda_\varphi$;
	see \cite[Examples~2.4 and 9.6]{deJ77}) if and only if $X/X_a$ is an $AM$-space (or, which is one thing, $X^s$ is an $AL$-space),
	that is to say,
	$$\distE{ \sup\{ x, y \}, X_a} = \max\{ \distE{x, X_a}, \distE{y, X_a} \} \text{ for } x, y \in X_+.$$
	
	From this perspective, we will show that if $X$ is a rearrangement invariant Banach space, then the quotient space $X/X_a$ is,
	in a sense, also \enquote{rearrangement invariant} (Theorem~\ref{fgwiazdka}). This result will turn out to be crucial in the next section,
	where we will deal with abstract Ces\`aro spaces looking for isometric copies of $\ell_\infty$. But let's start with some
	auxiliary lemmas.

	\begin{lemma} \label{lemma 1 distance}
		{\it Let $X$ be a Banach ideal space on $(\Omega, \Sigma, \mu)$ and let $J$ be an order ideal of $X$. If $f \in X$ then we have
			the following formula}
			\begin{equation*} \label{Lemma distance to an order ideal}
				\distE{f, J} = \inf\{ \norm{f-g}_X \colon g\in J \text{ {\it with} } \abs{g} \leqslant \abs{f} \}.
			\end{equation*}
	\end{lemma}
	\begin{proof}
		Without loss of generality we can assume that $f > 0$. Of course, to show the above equality we need only to prove that
		\begin{equation*}
			\inf\{ \norm{f-g}_X \colon g\in J \text{ such that } 0 \leqslant g \leqslant f \} \leqslant \distE{f, J}.
		\end{equation*}
		Let $\{ f_n \}_{n=1}^\infty \subset J^+$ be a sequence realizing the distance, that is,
		$\lim_{n\rightarrow\infty}\norm{f - f_n}_X = \distE{f, J}$.
		Set
		\begin{equation*}
			\Omega_n = \{ x \in \Omega \colon f_n(x) \leqslant f(x) \} \text{ for } n \in \mathbb{N},
		\end{equation*}
		and put
		\begin{equation*}
			\widetilde{f}_n = f_n\chi_{\Omega_n} + f\chi_{\Omega \setminus \Omega_n} \text{ for } n \in \mathbb{N}.
		\end{equation*}
		Observe that $\widetilde{f}_n \leqslant f$ for every $n\in\mathbb{N}$. Consequently,
		\begin{equation*}
			\norm{f - \widetilde{f}_n}_X = \norm{(f - f_n)\chi_{\Omega_n}}_X \leqslant \norm{\abs{f - f_n}}_X = \norm{f - f_n}_X.
		\end{equation*}
		Taking limits on both sides of the above inequality, we see immediately that
		\begin{equation*}
			\lim\limits_{n\rightarrow\infty}\norm{f - \widetilde{f}_n}_X \leqslant \lim\limits_{n\rightarrow\infty}\norm{f - f_n}_X = \distE{f, J}.
		\end{equation*}
		Moreover, since $J$ is an order ideal of $X$, it follows that $\{ \widetilde{f}_n \}_{n=1}^\infty\subset J$ and the proof is finished.
	\end{proof}

	The immediate conclusion from the above lemma is the following

	\begin{corollary}\label{colnierownosc}
		{\it Let $X$ be a Banach ideal space on $(\Omega, \Sigma, \mu)$ and let $J$ be an order ideal of $X$. If $0 \leqslant g \leqslant f \in X$ then}
		\begin{equation*}
			\distE{g, J} \leqslant \distE{f, J}.
		\end{equation*}
	\end{corollary}
	\begin{proof}
		Take $0 < g \leqslant f \in X$. Let $\{ f_n \}_{n=1}^\infty \subset J$ be a sequence that realizes the distance, that is to say,
		$\lim_{n\rightarrow\infty} \norm{f - f_n} = \distE{f,J}$. It follows from Lemma \ref{lemma 1 distance} that we can additionally
		assume that $f_n \leqslant f$ for every $n\in\mathbb{N}$. Put
		\begin{equation*}
		g_n = \min\{ g, f_n \} \text{ for } n \in \mathbb{N}.
		\end{equation*}
		Since $J$ is an order ideal of $X$, it is clear that $\{ g_n \}_{n=1}^\infty \subset J$. Observe also that $0 < g - g_n \leqslant f - f_n$
		for all $n \in \mathbb{N}$, so due to the ideal property of $X$ we get that
		$$
		\norm{g - g_n}_X \leqslant \norm{f - f_n}_X.
		$$
		But this simply means that $\distE{g, J} \leqslant \distE{f, J}$.
	\end{proof}

	\begin{lemma} \label{modul}
		{\it Let $X$ be a Banach ideal space on $(\Omega, \Sigma, \mu)$ and let $J$ be an order ideal of $X$. If $f \in X$ then}
		\begin{equation*}
			\distE{f, J} = \distE{\abs{f}, J}.
		\end{equation*}
	\end{lemma}
	\begin{proof}
		To show one of the inequalities let's take $f\in X$ and notice that
		\begin{align*}
		\distE{f, J} \geqslant \inf\{\norm{\abs{f} - \abs{g}}_X \colon g \in J \}
					   \geqslant \inf\{\norm{\abs{f} - h}_X \colon h \in J \} = \distE{\abs{f}, J}.
		\end{align*}
		
		On the other hand, for each $g \in J_{+}$ there is $h \in J$ such that
		$\norm{f - h}_X \leqslant \norm{\abs{f} - g}_X$. Indeed, for $g \in J_{+}$ set
		\[
		h \left( x \right) =\left\{ 
		\begin{array}{ccc}
		g\left( x \right)  & \text{if} & f\left( x\right) \geqslant 0, \\ 
		-g\left( x \right)  & \text{if} & f\left( x\right) < 0.
		\end{array}
		\right. 
		\]
		Then $h \in J$ and $\norm{f - h}_X = \norm{\abs{f} - g}_X$, whence $\distE{\abs{f}, J_{+}} \geqslant \distE{f, J}$.
		Since $\distE{\abs{f}, J} = \distE{\abs{f}, J_{+}}$ we conclude that $\distE{f, J} \leqslant \distE{\abs{f}, J}$
		and this finishes the proof.
	\end{proof}

	It will turn out right away that the formula expressing the distance of the non-increasing rearrangement of $f$ from the ideal $X_a$
	can be given in a fairly computable form.
	
	\begin{theorem} \label{thm: bardziej konstruktywna odleglosc}
		{\it Let $X$ be a rearrangement invariant space on $\Omega$, where $\Omega = (0,1)$, $\Omega = (0,\infty)$ or $\Omega = \mathbb{N}$,
			such that the ideal $X_a$ is non-trivial. Then for $f \in X$ we have the following equality}
		\begin{equation*}
			\distE{f^*,X_a} = \lim\limits_{n \rightarrow \infty} \norm{f^* \chi_{\Omega_n}}_X,
		\end{equation*}
		{\it where $\Omega_n = \Omega \cap \left((0, \frac{1}{n}) \cup (n,\infty)\right)$.}
	\end{theorem}
	\begin{proof}
		We will only give a proof when $\Omega = (0, \infty)$, because the other cases are completely analogous.
		
		To begin with, due to the fact that the sequence $\{ \norm{f^* \chi_{(0,1/n) \cup (n,\infty)}}_X \}_{n=1}^{\infty}$ is non-increasing
		and bounded from below by 0, the limit of the above sequence must exists. Moreover, it is clear that we can assume that $X \neq X_a$,
		because otherwise there is nothing to prove.
		
		Denote
		\begin{equation*}
		L = \lim\limits_{n \rightarrow \infty} \norm{f^* \chi_{(0,1/n) \cup (n,\infty)}}_X.
		\end{equation*}
		
		First of all, it is not hard to see that
		\begin{equation*}
		\distE{f^*,X_a} \geqslant L,
		\end{equation*}
		or, equivalently, $\norm{f^* - g}_X \geqslant L$ for each $g \in X_a$. Indeed,
		\begin{align*}
		\norm{f^* - g}_X & \geqslant \lim\limits_{n \rightarrow \infty} \norm{(f^* - g) \chi_{(0,1/n) \cup (n,\infty)}}_X \\
		& \geqslant \lim\limits_{n \rightarrow \infty} \abs{\norm{f^* \chi_{(0,1/n) \cup (n,\infty)}}_X - \norm{g \chi_{(0,1/n) \cup (n,\infty)}}_X} = L,
		\end{align*}
		where the equality follows from the fact that $g \in X_a$.
		
		On the other hand, our assumption that $X_a \neq \{ 0 \}$ implies that $\{ f^* \chi_{(1/n,n)} \}_{n=1}^{\infty} \subset X_a$
		(cf. \cite[Lemma~7.3~(ii)]{deJ77} or \cite[Theorem~B]{KT17}), so
		\begin{equation*}
		\distE{f^*,X_a} = \inf\{\norm{f^* - g} \colon g \in X_a \} \leqslant \lim\limits_{n \rightarrow \infty} \norm{f^* - f^* \chi_{(1/n,n)}}_X = L,
		\end{equation*}
		and the proof is done.
	\end{proof}

	Now we prove the main result of this section.
	
	\begin{theorem}\label{fgwiazdka}
		{\it Let $X$ be a rearrangement invariant space on $\Omega$, where $\Omega = (0,1)$ or $\Omega = (0,\infty)$, with the Fatou property.
			Then for $f \in X$ we have the following equality}
		\begin{equation*}
			\distE{f^*, X_a} = \distE{f, X_a}.
		\end{equation*}
	\end{theorem}
	\begin{proof}
		We can assume that $X_{a} \neq \left\{ 0 \right\}$, $X_{a} \neq X$ and $f \notin X_{a}$, because otherwise there is nothing to prove.
		Moreover, in view of Lemma \ref{modul}, we can also assume that $f > 0$.
		
		Let's start by noticing that the inequality
		\begin{equation}\tag{$\heartsuit$}
			\distE{f, X_a} \geqslant \distE{f^*, X_a}
		\end{equation}
		is true even without additional assumptions
		on the function $f \in X$. In fact, we have
		\begin{align*}
			\distE{f, X_a} & = \inf\{\norm{f - g}_X \colon g \in X_a \}
					   		\geqslant \inf\{\norm{f^* - g^*}_X \colon g \in X_a \} \\
					   	   & = \inf\{\norm{f^* - g}_X \colon g \in X_a \text{ and } g = g^* \}
					   		\geqslant \distE{f^*, X_a},
		\end{align*}
		where the first inequality follows essentially from the Calderón--Ryff theorem (see \cite[Theorem~2.10, p.~114]{BS88}; cf. also Lemma 4.6 in \cite[p.~95]{KPS82})
		and the second equality follows from Lemma 2.6 in \cite{CKP14}.
		Consequently, it remains to prove the reverse inequality. To do this, we divide the proof into several parts.
		
		$1^{\circ}$. Suppose that $f^{\ast}\left( \infty \right) = 0$. Then, as a consequence of Ryff's theorem (see \cite[Corollary 7.6, p. 83]{BS88}),
		there is a measure preserving transformation $\sigma$ from the support of $f$ onto the support of $f^*$ such that
		$f^{\ast}\circ \sigma = f$ $m$-almost everywhere on the support of $f$. If we were able to show that for each $g \in X_{a}$
		there exists $h_g \in X_{a}$ such that $\norm{f - h_g}_X \leqslant \norm{f^* - g}_X$, the proof of this part would be complete.
		For this, let's take $g \in X_{a}$ and put $h_g = g \circ \sigma$. Due to \cite[Lemma 2.6]{CKP14} and the fact that
		$h_g^* = (g \circ \sigma)^*$, we conclude that $h_g \in X_a$. Now, since we have the equalities
		\begin{equation*}
			\norm{f^* - g}_X = \norm{(f^* - g) \circ \sigma}_X = \norm{f - h_g}_X,
		\end{equation*}
		it follows that $\distE{f^*,X_a} \geqslant \distE{f,X_a}$. Combining this inequality with the inequality ($\heartsuit$), we see that we have just
		finished the proof in the first case.
		
		$2^{\circ}$. Moving on to the second case, let's assume that $f^*(\infty) > 0$. Denote
		\begin{equation*}
			\Omega_{\Box} = \{x \in \Omega \colon f(x) \Box f^*(\infty) \},
		\end{equation*}
		where $\Box$ is one of the following relations: $\geqslant$, $>$, $=$, $<$ or $\leqslant$ (for example, $\Omega_> = \{x \in \Omega \colon f(x) > f^*(\infty)\}$).
		Let us consider three subcases.
		
		Case 1: Assume that $m(\Omega_>) = \infty$. Then there exists a measure preserving transformation $\omega \colon \Omega_> \rightarrow [0,\infty)$
		such that $f^{\ast} \circ \omega = f$ $m$-almost everywhere on $\Omega_>$. Take $g \in X_a$ and put
		\begin{equation*}
			h_g = (g \circ \omega)\chi_{\Omega_>}.
		\end{equation*}
		We claim that $d_{f\chi_{\Omega_>} - h_g} = d_{f - h_g}$. Indeed, for all $0 \leqslant x < f^*(\infty)$ we have
		\begin{equation*}
			d_{f\chi_{\Omega_>} - h_g}(x) = \infty = d_{f - h_g}(x),
		\end{equation*}
		because $h_g \in X_a$, whence $h^*(\infty) = 0$. Furthermore, for $x \in \Omega_{\leqslant}$
		\begin{equation*}
			h_g(x) = 0 \text{ and } f(x) \leqslant f^*(\infty),
		\end{equation*}
		so for each $x \geqslant f^*(\infty)$ we obtain that
		\begin{align*}
			d_{f - h_g}(x) & = m\{t \in \Omega \colon \abs{f(t) - h_g(t)} > x \} = m\{t \in \Omega_> \colon \abs{f(t) - h_g(t)} > x \} \\
						   & = m\{t \in \Omega_> \colon \abs{f(t)\chi_{\Omega_>} - h_g(t)} > x \} = d_{f\chi_{\Omega_>} - h_g},
		\end{align*}
		and the claim follows. In consequence, we have
		\begin{align*}
			\norm{f^* - g}_X & = \norm{(f^* - g) \circ \omega}_X = \norm{f\chi_{\Omega_>} - h_g}_X \\
							 & = \norm{(f\chi_{\Omega_>} - h_g)^*}_X = \norm{(f - h_g)^*}_X = \norm{f - h_g}_X.
		\end{align*}
		
		Case 2: Suppose now that $m(\Omega_>) < \infty$ and $f\chi_{\Omega_{>}} \in X_a$. Put
		\begin{equation*}
		\Gamma_f = \{ g \in X_a \colon g(x) = f(x) \text{ for } x \in \Omega_{>} \text{ and } 0 \leqslant g(x) \leqslant f^*(\infty) \text{ otherwise } \}.
		\end{equation*}
		It follows directly from the above definition that $f\chi_{\Omega_{>}} \in \Gamma_f$ and $\Gamma_f \subset X_a$. We claim that
		\begin{equation} \tag{$\spadesuit$}
		\distE{f,X_a} = \distE{f,\Gamma_f}.
		\end{equation}
		To see this, for each $g \in X_a$ let us define a function $g_f$ in the following way
			\[
			g_f(x) =\left\{ 
			\begin{array}{ccc}
			f\left( x\right)  & \text{if} & x\in \Omega_{>} \\ 
			f\left( x\right)  & \text{if} & x\in \Omega \setminus \Omega_{>} \text{ and } g(x) > f^*(\infty) \\ 
			g\left( x\right)  &  & \text{otherwise}
			\end{array}
			\right. .
			\]
		Now, it is clear that $\norm{f - g_f}_X \leqslant \norm{f - g}_X$ and, due to the fact that
		$m\{x \in \Omega \setminus \Omega_{>} \colon g(x) > f^*(\infty) \} < \infty$, also $g_f \in \Gamma_f$. But this means that $\distE{f,X_a} \geqslant \distE{f,\Gamma_f}$.
		Since $\Gamma_f \subset X_a$, the opposite inequality is evident and the equality $(\spadesuit)$ is proved. Consequently, since $g \in X_a$, so
		\begin{equation*}
			\distE{f,X_a} = \inf\{\norm{f - g} \colon g \in \Gamma_f \} = \norm{f^*\chi_{(m(\Omega_{>}), \infty)}}_X = f^*(\infty)\norm{\chi_{(0,\infty)}}_X.
		\end{equation*}
		On the other hand, since $m(\Omega_>) < \infty$ and $f\chi_{\Omega_{>}} \in X_a$, applying Theorem \ref{thm: bardziej konstruktywna odleglosc},
		we conclude that
		\begin{equation*}
			\distE{f^*, X_a} = \lim\limits_{n \rightarrow \infty}\norm{f^*\chi_{(0,1/n) \cup (n,\infty)}} = f^*(\infty)\norm{\chi_{(0,\infty)}}_X,
		\end{equation*}
		that is to say, $\distE{f,X_a} = \distE{f^*,X_a}$.
		
		Case 3: Let $m(\Omega_>) < \infty$ and $f\chi _{\Omega_{>}} \in X\setminus X_{a}$.
		Observe that $f^*(0^+) = \infty$, because otherwise, keeping in mind that $m(\Omega_{>}) < \infty$,
		we would get $f\chi _{\Omega_{>}} \in X_a$ which is undoubtedly a contradiction. We will show that
		\begin{equation}\tag{$\diamondsuit$}
		\distE{f, X_a} = L \coloneqq \lim\limits_{n \rightarrow \infty} \norm{f^*\chi_{(0,\frac{1}{n})} + f^*(\infty)\chi_{(\frac{1}{n}, \infty)}}_X.
		\end{equation}
		Take $n_0 \in \mathbb{N}$ satisfying $f^*(1/n_0) > f^*(\infty)$ and note that this limit exists because the sequence
		$\left\{ \norm{f^*\chi_{(0,\frac{1}{n})} + f^*(\infty)\chi_{(\frac{1}{n}, \infty)}}_X \right\}_{n=1}^\infty$ is non-increasing for $n \geqslant n_0$
		and bounded from below. Now, there is a measure preserving transformation $\tau \colon \Omega_{>} \rightarrow (0,m(\Omega_{>}))$ such that
		$f^* \circ \tau = f$.
		Put $T_n = \tau^{-1}((0,1/n))$ for $n = n_0, n_0+1, ...$ and denote
		\begin{equation*}
		f_n = f^*\left( \frac{1}{n_0} \right)\chi_{T_n}
		\end{equation*}
		It is clear that $f_n \in X_a$ for $n \geqslant n_0$. Moreover, since $f^*(0^+) = \infty$, so
		\begin{equation*}
		\left( f^* - f^*\left( \frac{1}{n_0} \right) \right)\chi_{(0,\frac{1}{n})} \geqslant f^*(\infty)\chi_{(0,\frac{1}{n})} \text{ for } n \in \mathbb{N} \text{ big enough},
		\end{equation*}
		and, consequently, we have
		\begin{equation*}
		\distE{f,X_a} \leqslant \lim\limits_{n \rightarrow \infty} \norm{f - f_n}_X
			= \lim\limits_{n \rightarrow \infty} \norm{\left( f^* - f^*\left( \frac{1}{n_0} \right) \right)\chi_{(0,\frac{1}{n})} + f^*(\infty)\chi_{(\frac{1}{n}, \infty)}}_X
			\leqslant L.
		\end{equation*}
		On the other hand, we get
		\begin{align*}
		L \geqslant \distE{f,X_a} & = \inf\{\norm{f - g}_X \colon g \in X_a \text{ and } \abs{g} \leqslant \abs{f} \}\\
		  						  &\geqslant \inf\{\norm{f^* - g^*}_X \colon g \in X_a \text{ and } \abs{g} \leqslant f \} \\
		  						  & \geqslant \inf\{\norm{f^* - g^*}_X \colon g \in X_a \text{ and } 0 \leqslant g^* \leqslant f^* \},
		\end{align*}
		where the first equality follows from Lemma~\ref{lemma 1 distance} and second inequality follows from the Calder\'{o}n--Ryff theorem.
		Thus, to finish the proof of $(\diamondsuit)$, it is enough to show that
		\begin{equation} \tag{$\clubsuit$}
		\text{for every } \varepsilon > 0 \text{ and } h \in X_a \text{ with } 0 \leqslant h^* \leqslant f^* \text{ we have } \norm{f^* - h^*}_X \geqslant L - \varepsilon.
		\end{equation}
		To see this, take $\varepsilon > 0$ and $h \in X_a$ as above. We can find a number, say $m \in \mathbb{N}$, such that
		\begin{equation*}
		\norm{h^*\chi_{(0,\frac{1}{m}) \cup (m,\infty)}} < \epsilon,
		\end{equation*}
		because $h^* \in X_a$ (see \cite[Lemma~2.6]{CKP14}). Hence,
		\begin{align*}
		\norm{f^* - h^*}_X & = \norm{f^* - h^*\chi_{(\frac{1}{m}, m)} - h^*\chi_{(0,\frac{1}{m}) \cup (m,\infty)}}_X \\
						   & \geqslant \norm{f^* - h^*\chi_{(\frac{1}{m}, m)}}_X - \norm{h^*\chi_{(0,\frac{1}{m}) \cup (m,\infty)}}_X \\
						   & > \norm{f^* - h^*\chi_{(\frac{1}{m}, m)}}_X - \varepsilon \\
						   & \geqslant \norm{f^*\chi_{(0,\frac{1}{m}) \cup (m,\infty)} + (f^* - h^*)\chi_{(\frac{1}{m}, m)}}_X - \varepsilon
						   \geqslant \norm{f^*\chi_{(0,\frac{1}{m}) \cup (m,\infty)}}_X - \varepsilon,
		\end{align*}
		and $(\clubsuit)$ follows. But this means that the proof of $(\diamondsuit)$ is finished as well. However, using Theorem
		\ref{thm: bardziej konstruktywna odleglosc} again, we see that
		\begin{equation*}
			\distE{f^*,X_a} = \lim\limits_{n \rightarrow \infty} \norm{f^*\chi_{(0,1/n) \cup (n,\infty)}}_X,
		\end{equation*}
		that is, $\distE{f,X_a} = \distE{f^*,X_a}$. But this, finally, completes the proof of the last case and, at the same time, also of the whole theorem.
	\end{proof}

	\begin{theorem}
		{\it Let $X$ be a rearrangement invariant space on $\mathbb{N}$ with the Fatou property. Then for $f \in X$ we have the following equality}
		\begin{equation*}
			\distE{f^*, X_a} = \distE{f, X_a}.
		\end{equation*}
	\end{theorem}
	\begin{proof}
		The argument in the sequence case (which is probably not a suprise) is analogous to the proof of Theorem \ref{fgwiazdka} and,
		in fact, it is even easier. But we want to give a few clues that (we hope) will allow to easily modify the mentioned proof.
		
		$1^{\circ}$. For a rearrangement invariant sequence space $X$ the ideal $X_a$ is always non-trivial.
		
		$2^{\circ}$. Since $X \hookrightarrow \ell_\infty$ (cf. \cite[Corollary~6.8, p.~78]{BS88}), so $X_a \hookrightarrow c_0$
		and this means that $f^*(\infty) = 0$, whenever $f \in X_a$.
		
		$3^{\circ}$. Every function supported on the set of finite measure belongs to $X_a$.
		
		$4^{\circ}$. The inequality (coming from the Calder\'{o}n--Ryff theorem):
		\begin{equation*}
		 \norm{f^* - g^*}_X \leqslant \norm{f - g}_X,
		\end{equation*}
		works also for every rearrangement invariant sequence space $X$ (cf.~\cite[Corollary~7.6, p.~83]{BS88}).
	\end{proof}

	\begin{corollary} \label{corollary: formula na odleglosc}
		{\it Let $X$ be a rearrangement invariant space on $\Omega$, where $\Omega = (0,1)$, $\Omega = (0,\infty)$ or $\Omega = \mathbb{N}$,
			such that the ideal $X_a$ is non-trivial. Then for $f \in X$ we have the following equality}
		\begin{equation*}
			\distE{f,X_a} = \lim\limits_{n \rightarrow \infty} \norm{f^* \chi_{\Omega_n}}_X,
		\end{equation*}
		{\it where $\Omega_n = \Omega \cap \left((0, \frac{1}{n}) \cup (n,\infty)\right)$.}
	\end{corollary}
	
	The above corollary can be interpreted in the following way: the function $f \in X$ can be approximated by a bounded function supported
	in set of finite measure to within $\distE{f,X_a} + \varepsilon$, where $\varepsilon > 0$ can be arbitrarily small, and this function
	can be selected in an explicit way. A similar result, but for Orlicz spaces equipped with an Orlicz norm, can be found in \cite[Lemma~10.1, p.~84]{KR61}.

	\begin{corollary}[E. de Jonge, 1977 \cite{deJ77}] \label{corollary: de Jonge}
		{\it Let $X$ be a rearrangement invariant space on $(0,\infty)$ with the Fatou property such that $L_\infty \hookrightarrow X$.
			Suppose that}
		\begin{equation}\tag{${\mathscr D}_\infty$}
			\forall f \in X \text{ } \forall A \subset (0,\infty) \textit{ with } m(A) < \infty \colon f\chi_A \in X_a.
		\end{equation}	
		{\it Then for all $f \in X$ we have the following formula}
		\begin{equation*}
			\distE{f,X_a} = f^*(\infty)\varphi_X(\infty).
		\end{equation*}
	\end{corollary}

	Before we provide the proof, let us note that the above condition (${\mathscr D}_\infty$) itself plays in rearrangement invariant
	spaces an analogous role to the $\Delta_2(\infty)$-condition in the case of Orlicz spaces. In fact, if $X$ is an Orlicz space $L_F$
	on $(0,\infty)$ generated by the Orlicz function $F$ which vanishes outside zero and $F \in \Delta_2(\infty)$,
	then $L_\infty \hookrightarrow X$ and $X \in ({\mathscr D}_\infty)$. However, the assumptions of the above
	corollary are true also in the case of Lorentz spaces $\Lambda_\varphi$ provided $\varphi(0^+) = 0$ and $\varphi(\infty) < \infty$.

	\begin{proof}
		Suppose that $X$ is as above and observe that the sequence $\left\{(f^*\chi_{(n,\infty)})^*\right\}_{n=1}^\infty$
		is bounded and converges uniformly to a bounded function $f^*(\infty)\chi_{(0,\infty)}$, so in view of the embedding
		$L_\infty \hookrightarrow X$ we get
		\begin{equation*}
		\distE{f,X_a} = \lim\limits_{n\rightarrow\infty} \norm{f^*\chi_{(n,\infty)}}_X
		= \lim\limits_{n\rightarrow\infty} \norm{(f^*\chi_{(n,\infty)})^*}_X = \norm{f^*(\infty)\chi_{(0,\infty)}}_X
		= f^*(\infty)\varphi_X(\infty),
		\end{equation*}
		where the first equality follows from Corollary \ref{corollary: formula na odleglosc}.
	\end{proof}
	
	Although we believe that the below results (Corollaries \ref{corollary: 1} and \ref{corollary: 2}) are expected, we will present proofs
	of these facts as illustrations of the techniques we developed in this paragraph.

	\begin{corollary} \label{corollary: 1}
		{\it Let $\varphi$ be a quasi-concave function such that $\lim_{t \rightarrow 0^{+}} t/\varphi(t) = 0$. Then the Marcinkiewicz space $M_\varphi$
			contains a lattice isometric copy of $\ell_\infty$.}
	\end{corollary}
	\begin{proof}
		Let us consider the function $\psi$ defined in the following way
		\begin{equation*}
			\psi(0) = 0 \text{ and } \psi(t) = \frac{t}{\varphi(t)} \text{ for } t > 0.
		\end{equation*}
		It follows from \cite[Theorem~5.2, p.~66]{BS88} that $\psi$ is also a quasi-concave function as a fundamental function of the space $X'$.
		In addition, due to the condition $\lim_{t \rightarrow 0^{+}} t/\varphi(t) = 0$, the function $\psi$ is continuous (cf. \cite[Corollary~5.3, p.~67]{BS88}).
		Therefore, we can find $0 < a < 1$ such that the function $\psi$ is strictly increasing on $(0,a)$. Consequently, the function $\psi'$
		is non-increasing and positive $m$-almost everywhere on $(0,a)$. Put $g = \psi'\chi_{(0,a)}$. Then
		\begin{equation*}
			\frac{\varphi(t)}{t}\int_0^t g^*(s)ds = \frac{\varphi(t)}{t}\int_0^t g(s)ds = \frac{\varphi(t)}{t}\psi(t) = 1 \text{ for } 0 < t < a,
		\end{equation*}
		and, because the function $t \mapsto \varphi(t)/t$ is non-increasing, we also have that
		\begin{equation*}
			\frac{\varphi(t)}{t}\int_0^t g^*(s)ds = \frac{\varphi(t)}{t} \frac{a}{\varphi(a)} \leqslant 1 \text{ for } t \geqslant a.
		\end{equation*}
		This means that $\norm{g}_{M_\varphi} = 1$. Moreover, we can repeat the above argument with the same result for the function
		$\psi_\varepsilon = \psi\chi_{(0,\varepsilon)}$, where $0 < \varepsilon < a$, in place of the function $\psi$. Therefore, using
		Corollary~\ref{corollary: formula na odleglosc}, we get
		\begin{equation*}
		d_{M_\varphi}(g, (M_\varphi)_a) \geqslant \lim_{n\rightarrow\infty}\norm{g^*\chi_{(0,\frac{1}{n})}}_{M_\varphi}
														= \lim_{n\rightarrow\infty}\norm{g\chi_{(0,\frac{1}{n})}}_{M_\varphi}
														\geqslant \lim_{n\rightarrow\infty}\norm{\psi_{\varepsilon_n}}_{M_\varphi} = 1,
		\end{equation*}
		for a properly choosen subsequence $\{\varepsilon_n\}_{n=1}^\infty$.
	\end{proof}

	\begin{lemma} \label{lemma: L_infty w X, to kopia}
		{\it Let $X$ be a rearrangement invariant space on $(0,\infty)$. If $L_\infty \hookrightarrow X$, then $X$ contains a lattice isometric copy of $\ell_\infty$.}
	\end{lemma}

	Let us note in addition that the converse is not true, for example, the Zygmund space $e^L$ (recall that $e^L = L_F$, where $F(t) = e^t -1$)
	contains a lattice isometric copy of $\ell_\infty$, but it is clear that $L_\infty \not\hookrightarrow e^L$.
	
	\begin{proof}
		Take $f = \chi_{(0,\infty)}/\norm{\chi_{(0,\infty)}}_X$ and let $\{N_i\}_{i=1}^\infty$ be a pairwise disjoint family of an infinite
		subsets of $\mathbb{N}$ such that $\bigcup_{i=1}^\infty N_i = \mathbb{N}$. Moreover, let $\{\omega_j\}_{j=1}^\infty$ be a sequence of real numbers
		such that $0 = \omega_1 < \omega_2 < ...$ and $\lim_{j \rightarrow \infty} \omega_j = \infty$. Put $\Omega_j^i = (\omega_j,\omega_{j+1})$
		for $i \in \mathbb{N}$ and $j \in N_i$. Consider the following sequence $\left\{f\chi_{\bigcup_{j \in N_i} \Omega_j^i}\right\}_{i=1}^\infty$. It is clear that
		this sequence is pairwise disjoint, $\norm{\sum_{i=1}^\infty f\chi_{\bigcup_{j \in N_i} \Omega_j^i}}_X = 1$ and
		$\norm{f\chi_{\bigcup_{j=1}^\infty \Omega_j^i}}_X = 1$, because $\left(f\chi_{\bigcup_{j \in N_i} \Omega_j^i}\right)^* = f$ for every $i \in \mathbb{N}$
		and $\norm{f}_X = 1$. Consequently, in order to complete the proof we can refer to \cite[Theorem~1]{Hu98}.
	\end{proof}

	\begin{corollary} \label{corollary: 2}
		{\it Let $\varphi$ be a concave function on $(0,\infty)$ such that $\varphi(0^{+}) = 0$. Then the Lorentz space $\Lambda_\varphi$ contains
			a lattice isometric copy of $\ell_\infty$ if and only if $\varphi(\infty) < \infty$.}
	\end{corollary}
	\begin{proof}
		It follows from \cite[Lemma~5.1]{KPS82} that the space $\Lambda _{\varphi}$ is order continuous if $\varphi(\infty) = \infty$.
		Therefore, due to the well-known Lozanovski\u{\i} result, it does not contain even an order isomorphic copy of $\ell _{\infty}$
		and proves the necessity.
		
		Now, suppose that $\varphi(\infty) < \infty$. But this means that $L_\infty \hookrightarrow X$ and it is enough to use Lemma~\ref{lemma: L_infty w X, to kopia}.
	\end{proof}
	
	Let us add that it may happen that the Lorentz space $\Lambda_\varphi$ contains an order isomorphic copy of $\ell_\infty$ but not an order
	isometric copy of $\ell_\infty$. Indeed, if $\varphi(0^+) > 0$ and $\varphi(\infty) = \infty$, then the space $\Lambda_\varphi$ is not
	order continuous but is strictly monotone. 

	\section{On isometric copies of $\ell_\infty$ in abstract Ces\`aro spaces}
	
	For a Banach function space on $(0,1)$ or $(0,\infty)$ we define the {\it abstract Ces\`aro function space $CX$} as a set
	\begin{equation*}
		CX = \{f \in L_0 \colon C(\abs{f}) \in X \} \text{ with the natural norm } \norm{f}_{CX} = \norm{C(\abs{f})}_X,
	\end{equation*}
	where $C$ denotes the {\it continuous Ces\`aro operator} defined as
	\begin{equation*}
		C \colon f \mapsto C(f)(x) = \frac{1}{x} \int_0^x \abs{f(t)} dt \text{ for } x > 0.
	\end{equation*}
	
	However, in the sequence case, by the {\it discrete Ces\`aro operator} we will understand the operator $C_d$ defined as
	\begin{equation*}
		C_d \colon \{x_n\}_{n=1}^\infty \mapsto \left\{ \frac{1}{n} \sum_{k=1}^n x_k \right\}_{n=1}^\infty,
	\end{equation*}
	and then the corresponding {\it abstract Ces\`aro sequence space $CX$} is defined in an analogous way as above, that is to say,
	as a set of all sequences $x = \{x_n\}_{n=1}^\infty$ such that $C_d(x) \in X$ for which the following norm
	\begin{equation*}
		\norm{x}_{CX} = \norm{C_d(\abs{x})}_X
	\end{equation*}
	is finite
	
	Because it always should be clear from the context which Ces\`aro construction we use, we will abuse the notation
	a little by writing simply $C$ for the continuous as well as the discrete Ces\'aro operator. Moreover, to avoid
	some trivialities and unnecessary complications, let's assume that we are only interested in non-trivial Ces\`aro spaces.
	Incidentally, if the Ces\`aro operator $C$ is bounded on $X$, then $X \hookrightarrow CX$ and, consequently, the space $CX$
	is non-trivial.
	
	The classical Ces\`aro spaces $ces_p$ and $Ces_p$ are only a special case of this general construction $CX$ when $X = \ell_p$
	or $X = L_p$, respectively. We will not say more about the history and the current development of the Cesaro space theory,
	referring to \cite{AM14} and the references given there (see also \cite{Be96}, \cite{AM09}, \cite{ALM17}, \cite{CR16}, \cite{DS07},
	\cite{Ja74}, \cite{KK09}, \cite{KT17}, \cite{KKM21}, \cite{LM15} and \cite{NP10}; cf.~\cite{KMP07}).
	
	Let us now formulate and prove the main result of this section.

	\begin{theorem} \label{CX ma izometryczne l_infty}
		{\it Let $X$ be a rearrangement invariant Banach space on $(0,1)$, $(0,\infty)$ or $\mathbb{N}$ with the Fatou property such that
			the Ces\`aro operator $C$ is bounded on $X$ and the ideal $X_a$ is non-trivial. Suppose there exists an element $f \in X$
			and $0 \leqslant a < b \leqslant \infty$ such that} $\norm{f}_X = \distE{f,X_a} = 1$ {\it and} $\norm{f^*\chi_{(0,a)\cup(b,\infty)}}_{CX} = 1$.
			{\it Then $CX$ contains a lattice isometric copy of $\ell_\infty$}.
	\end{theorem}

	The proof of the above theorem makes strong use of Theorem~\ref{fgwiazdka} from \S3 and Theorem~16 from \cite{KT17}. Note that Theorem~16 from
	\cite{KT17} has been proved for rearrangement invariant function, but not sequence, spaces. For the missing proof, we refer to Appendix.

	\begin{proof}
		Let $f \in X$ be as above. Put $A = (0,a) \cup (b,\infty)$ and $A^c = (a,b)$. In view of Theorem \ref{fgwiazdka}, we have the equality
		\begin{equation*}
		\distE{f^*, X_a} = \distE{f, X_a} = 1.
		\end{equation*}
		Note also that $f^*\chi_{A^c} \in (CX)_a$ (see \cite[Lemma~8]{KT17} in the case of function spaces; however, simple modification of this
		argument works in the sequence case as well). Consequently, taking $h = f^*\chi_{A}$, we have
		\begin{align*}
		\distC{h, (CX)_a} & = \inf \{ \norm{h - g}_{CX} \colon g \in (CX)_a \text{ and } 0 < g \leqslant h \} \\
						  & = \inf \{ \norm{Ch - Cg}_{X} \colon g \in C(X_a) \text{ and } 0 < g \leqslant h \} \\
						  &		\quad\quad\quad (\text{because } h - g \geqslant 0 \text{ and } CX_a = C(X_a) \text{, see \cite[Theorem 16]{KT17}}) \\
						  & = \inf \{ \norm{C(h+f^*\chi_{A^c}) - C(g+f^*\chi_{A^c})}_{X} \colon C(g) \in X_a \text{ and } 0 < g \leqslant h \} \\
						  & \geqslant \inf \{ \norm{C(f^*) - C(g')}_{X} \colon g' \in CX \text{ and } C(g') \in X_a \} \\
						  &		\quad\quad\quad (\text{since } g+f^*\chi_{A^c} \in C(X_a) ) \\
						  & \geqslant \inf \{ \norm{C(f^*) - {g''}}_{X} \colon {g''} \in X_a \} \\
						  & 	\quad\quad\quad (\text{we take the infimum on a larger set}) \\
						  & = \distE{C(f^*), X_a} \geqslant \distE{f^*, X_a} = 1.
		\end{align*}
		where the last inequality follows from Corollary \ref{colnierownosc} and the fact that $f^* \leqslant C(f^*)$.
		Now, since $\supp((CX)_a) = \supp(X)$ (see, again, \cite[Lemma~8]{KT17}) we can apply Theorem \ref{Twierdzenie Hudzika} and finish the proof.
	\end{proof}

	It should be probably mentioned here that Theorem \ref{CX ma izometryczne l_infty} covers all previous results
	concerning the problem of existence of a lattice isometric copy of $\ell_\infty$ in Ces\`aro spaces $CX$. To be more precise, Kami\'nska and
	Kubiak \cite{KK09} considered this problem for $X = \ell_F$ and Kiwerski and Kolwicz \cite{KK18b} for $X = L_F$. A common feature of these
	results is a certain assumption about the Orlicz class, namely, about the closedness of the Orlicz class under the action of the Ces\`aro
	operator, that is to say,
	\begin{equation*}
		\left\{ \{x_n\}_{n=1}^\infty \in \ell_F \colon \sum_{n=1}^\infty F(x_n) < \infty \right\}
					\subset \left\{ \{x_n\}_{n=1}^\infty \in \ell_F \colon \sum_{n=1}^\infty F\left( \frac{1}{n} \sum_{k=1}^n x_k \right) < \infty \right\},
	\end{equation*}
	in the sequence case and the modular-type inequality of the form
	\begin{equation*}
		\int_0^\infty F\left(\frac{1}{x}\int_0^x\abs{f(t)}dt\right)dx \leqslant M\int_0^\infty F(\abs{f(x)})dx, 
	\end{equation*}
	in the function case (which looks a little like the boundedness\footnote{Orlicz class is not even a linear space,
	therefore is makes little sense to speak about such an object as bounded operator on it.}
	of the Ces\`aro operator on the Orlicz class or like the Hardy inequality with the Orlicz function $F$ instead of the power function; cf.~\cite[\S3]{KMP07}).
	Note that both assumptions imply the boundedness of the Ces\`aro operator on the space $\ell_F$ (see \cite[Proposition~2]{KK09}) and, respectively, $L_F$ (see \cite{KK18b}).
	Our Theorem~\ref{CX ma izometryczne l_infty}, at least formally, improves both of these results. However, it is not so easy to see that
	Theorem~\ref{CX ma izometryczne l_infty} does in fact imply the earlier results from \cite{KK09} and \cite{KK18b}. To show
	this, we will take an even more general point of view, because we will consider the Calder{\'o}n--Lozanovski{\u \i} spaces
	$X_F$ (recall, that the spaces $(L_1)_F$ and $(\ell_1)_F$ coincide, up to the equality of norms, with the Orlicz spaces $L_F$
	and $\ell_F$, respectively).
	
	Let us also remind that by the {\it Calder{\'o}n--Lozanovski{\u \i} class} we understand the set $\{f \in X_F \colon \norm{F(\abs{f})}_X < \infty\}$
	(cf.~\cite[\S3]{Ma89} for some comments and key properties in the particular case when $X = L_1$).
	
	\begin{corollary} \label{corollary: kopia w funkcyjnym Calderonie}
		{\it Let $X$ be an order continuous rearrangement invariant function space with the Fatou property and let $F$ be an Orlicz
			function. Assume that there exists $M > 0$ such that for all functions, say $f$, from the Calder{\'o}n--Lozanovski{\u \i} class
			we have the following inequality:
		}
			\begin{equation} \tag{${\mathscr M}$} \label{nierownosc modularna}
				\norm{F(C(\abs{f}))}_X \leqslant M\norm{F(\abs{f})}_X.
			\end{equation}
			{\it If $F$ does not satisfy the $\Delta_2^X$-condition\footnote{Let's just mention that the $\Delta_2^X$-condition is an appropriate
					modification of the well-known condition $\Delta_2$ from the Orlicz space theory. Roughly speaking, it is defined in such a way
					as to correspond to the order continuity of the Calder{\'o}n--Lozanovski{\u \i} space $X_F$. For a more precise definitions we refer,
					for example, to \cite[p.~642]{HKM96}.}, then the space $C(X_F)$ contains a lattice isometric copy of $\ell_\infty$.}
	\end{corollary}
	\begin{proof}
		Since $F \notin \Delta_2^X$, it follows from Theorem~1 and Theorem~2 in \cite{HKM96} that the Calder{\'o}n--Lozanovski{\u \i}
		space $X_F$ contains a lattice isometric copy of $\ell_\infty$. Therefore, in view
		of Theorem~\ref{Hudzik w druga strone} (let us note the fact that $\supp((X_F)_a) = \supp(X_F)$) and Theorem~\ref{fgwiazdka},
		there exists a function $f \in X_F$ such that
		\begin{equation*}
			\norm{f}_{X_F} = d_{X_F}(f, (X_F)_a) = d_{X_F}(f^*, (X_F)_a) = 1.
		\end{equation*}
		Because the space $X_F$ is rearrangement invariant, so $\norm{f^*}_{X_F} =1$ and, consequently, $\norm{F(f^*)}_X \leqslant 1$.
		Now, we divide the proof into two parts\footnote{It follows from Theorem~\ref{Hudzik w druga strone} that a rearrangement
			invariant space $X$ on $(0,\infty)$ with $\supp(X_a) = (0,\infty)$ contains a lattice isometric copy of $\ell_\infty$
			if and only if there is a function, say $f$, such that $\norm{f} = \distE{f,X_a} = 1$. In many situations (for example, if
			$X$ is a Calder{\'o}n--Lozanovski{\u \i} space) the function $f$ is given in a simple form, i.e., $f = \sum_{n=1}^\infty f_n$,
			where $f_n = a_n \chi_{\Omega_n}$, $a_n \in \mathbb{R}_{+}$, $\Omega_n$ are a pairwise disjoint measurable subsets of $(0,\infty)$
			and $f_n \in S(X)$ for $n \in \mathbb{N}$. In this case, however, we have essentially two possibilities, namely either $f$
			is bounded (and has a support of infinite measure) or $f$ is unbounded and has a support of a finite measure. Indeed,
			it is easy to see that
			\begin{equation*}
				1 \leqslant \distE{\sum_{k=1}^\infty f_{n_k}, X_a} \leqslant \distE{\sum_{n=1}^\infty f_{n}, X_a} = 1,
			\end{equation*}
			that is to say, the subsequence $\{f_{n_k}\}_{k=1}^\infty$ still builds a lattice isometric copy of $\ell_\infty$. Therefore,
			if we can find a subsequence, say $\{f_{n_k}\}_{k=1}^\infty$, which is uniformly bounded, then it is enough to put $f = \sum_{k=1}^\infty f_{n_k}$.
			Clearly, $f \in L_\infty$. On the other hand, if some subsequence $\{f_{n_k}\}_{k=1}^\infty$ is unbounded, then since
			$X \hookrightarrow L_1 + L_\infty$, so (taking further subsequence, if needed) we see that $\sum_{k=1}^\infty m(\Omega_{n_k}) < \infty$.
			}.
		
		$1^\circ$. Suppose that $f^*\chi_{(c,\infty)} \in (X_F)_a$ for some (equivalently, for all) $c > 0$. Consequently, we have that
		\begin{equation*}
			1 = d_{X_F}(f^*, (X_F)_a) = \lim\limits_{n \rightarrow \infty}\norm{f^*\chi_{(0,\frac{1}{n})\cup (n,\infty)}}_{X_F}
			  = \lim\limits_{n \rightarrow \infty}\norm{f^*\chi_{(0,\frac{1}{n})}}_{X_F},
		\end{equation*}
		where the second equality follows from Theorem~\ref{thm: bardziej konstruktywna odleglosc}. Since $X \in (OC)$, so we can find $b > 0$ with
		\begin{equation*}
			\norm{(F(f^*))\chi_{(0,b)}}_X \leqslant \frac{1}{M},
		\end{equation*}
		where the constant $M > 0$ is from the condition (\ref{nierownosc modularna}). Therefore, we have the following inequalities
		\begin{equation*}
			\norm{F(C(f^*\chi_{(0,b)}))}_X \leqslant M \norm{F(f^*\chi_{(0,b)})}_X = M\norm{(F(f^*))\chi_{(0,b)}}_X \leqslant 1.
		\end{equation*}
		But this means that $\norm{f^*\chi_{(0,b)}}_{C(X_F)} \leqslant 1$. What's more, thanks to our assumptions, $f^*\chi_{(b,\infty)} \in (X_F)_a$,
		whence
		\begin{equation*}
			1 = d_{X_F}(f^*, (X_F)_a) \leqslant \norm{f^* - f^*\chi_{(b,\infty)}}_{X_F}
			  = \norm{f^*\chi_{(0,b)}}_{X_F} \leqslant \norm{C(f^*\chi_{(0,b)})}_{X_F} \leqslant 1,
		\end{equation*}
		that is, $\norm{C(f^*\chi_{(0,b)})}_{X_F} = 1$. Now we can apply Theorem~\ref{CX ma izometryczne l_infty} and finish the proof of this part.
		
		$2^\circ$. Assume that $f^*\chi_{(c,\infty)} \in X_F \setminus (X_F)_a$ and $f^*\chi_{(0,c)} \in (X_F)_a$ for some (equivalently, for all) $c > 0$.
		Referring once again to Theorem~\ref{thm: bardziej konstruktywna odleglosc}, we have the following equalities
		\begin{equation*}
			1 = d_{X_F}(f^*, (X_F)_a) = \lim\limits_{n \rightarrow \infty}\norm{f^*\chi_{(0,\frac{1}{n})\cup (n,\infty)}}_{X_F}
			  = \lim\limits_{n \rightarrow \infty}\norm{f^*\chi_{(n,\infty)}}_{X_F}.
		\end{equation*}
		Just as before, since $X \in (OC)$, so
		\begin{equation*}
		\norm{(F(f^*))\chi_{(b,\infty)}}_X \leqslant \frac{1}{M} \text{ for some } b > 0 \text{ big enough},
		\end{equation*}
		where $M > 0$ is from the condition (\ref{nierownosc modularna}). Moreover, using the condition (\ref{nierownosc modularna}), we get
		\begin{equation*}
			\norm{F(C(f^*\chi_{(b,\infty)}))}_X \leqslant M \norm{F(f^*\chi_{(b,\infty)})}_X = M\norm{(F(f^*))\chi_{(b,\infty)}}_X \leqslant 1.
		\end{equation*}
		that is, $\norm{f^*\chi_{(b,\infty)}}_{C(X_F)} \leqslant 1$. Now, we claim that
		\begin{equation} \tag{$=$} \label{warunek rownosc}
			\norm{f^*\chi_{(b,\infty)}}_{C(X_F)} = 1.
		\end{equation}
		To prove (\ref{warunek rownosc}) it is enough to show that
		\begin{equation*}
			\norm{F(C((1+\varepsilon)f^*\chi_{(b,\infty)}))}_X > 1 \text{ for all } \varepsilon > 0.
		\end{equation*}
		Take $\varepsilon > 0$. Note that
		\begin{equation*}
			C\left((1+\varepsilon)f^*\chi_{(b,\infty)}\right)(x) = (1+\varepsilon) \frac{1}{x}\int_b^x f^*(t)dt
																			\geqslant (1+\varepsilon)\left(1 - \frac{b}{x}\right)f^*(x),
		\end{equation*}
		for $x \geqslant b$. Moreover, there exists $B > b$ such that
		\begin{equation*}
			(1+\varepsilon)\left(1 - \frac{b}{x}\right)f^*(x) \geqslant \left(1 + \frac{\varepsilon}{2}\right)f^*(x) \text{ for } x \geqslant B.
		\end{equation*}
		In consequence, we have
		\begin{equation*}
			\norm{F(C((1+\varepsilon)f^*\chi_{(b,\infty)}))}_X \geqslant \norm{F(C((1+\varepsilon)f^*\chi_{(B,\infty)}))}_X
									\geqslant \norm{F\left(\left(1+\frac{\varepsilon}{2}\right)f^*\chi_{(B,\infty)}\right)}_X > 1,
		\end{equation*}
		where the last inequality follows from the fact that $\norm{f^*\chi_{(B,\infty)}}_{X_F} = 1$ (to see this just note that due to our assumptions
		$f^*\chi_{(0,B)} \in (X_F)_a$ and then $1 = d_{X_F}(f^*, (X_F)_a) \leqslant \norm{f^* - f^*\chi_{(0,B)}}_{X_F} = \norm{f^*\chi_{(B,\infty)}}_{X_F} \leqslant 1$)
		and the definition of the Luxemburg--Nakano norm. But this proves the claim (\ref{warunek rownosc}).
		Again, we are left to use Theorem~\ref{CX ma izometryczne l_infty}. We have therefore completed the proof.
	\end{proof}

	\begin{corollary} \label{corollary: kopia w ciagowym calderonie}
		{\it Let $X$ be an order continuous rearrangement invariant sequence space with the Fatou property and let $F$ be an Orlicz
			function. Assume that the Calder{\'o}n--Lozanovski{\u \i} class is closed under the Ces\`aro operator $C$. If $F$ vanishes
			only at zero and does not satisfy the $\Delta_2(0)$-condition, then the space $C(X_F)$ contains a lattice isometric copy of $\ell_\infty$.}
	\end{corollary}
	\begin{proof}
		To begin with, since $X = X_a \hookrightarrow c_0 \hookrightarrow \ell_\infty$ and the Orlicz function $F$ vanishes only at 0,
		it follows that the condition $\Delta_2(0)$ is equivalent to the condition $\delta_2^X$ (see \cite[Lemma~4]{HN05})\footnote{We
			leave the definition of the condition $\delta_2^X$ due to the fact that it concerns the Musielak--Orlicz functions,
			which we will not consider here. At the same time, we refer for details to \cite[p.~525]{FH99}.}.
		Due to the lack of the $\delta_2^X$-condition and Lemma~2.4 from \cite{FH99}, the Calder{\'o}n--Lozanovski{\u \i} space $X_F$
		contains a lattice isometric copy of $\ell_\infty$ (note only that the results from \cite{FH99} have been proven for a wider
		than $X_F$ class of spaces, the so-called generalized Calder{\'o}n--Lozanovski{\u \i} spaces $X_M$, in which the Orlicz function
		$F$ is replaced by the much more general Musielak--Orlicz function $M$). Exactly as in the function' case
		(cf. Corollary~\ref{corollary: kopia w funkcyjnym Calderonie}), citing Theorem~\ref{Twierdzenie Hudzika} along with Theorem~\ref{fgwiazdka},
		we can find a sequence $f \in X_F$ with
		\begin{equation*}
			\norm{f}_{X_F} = d_{X_F}(f, (X_F)_a) = d_{X_F}(f^*, (X_F)_a) = 1.
		\end{equation*}
		What's more $\norm{f^*}_{X_F} =1$ and, consequently, $\norm{F(f^*)}_X \leqslant 1$. Now, because the Calder{\'o}n--Lozanovski{\u \i}
		class is closed under the Ces\`aro operator $C$ we conclude that $\norm{F(C(f^*))}_{X_F} < \infty$. Since $X \in (OC)$, there
		exists $n \in \mathbb{N}$ such that
		\begin{equation*}
			\norm{(F(C(f^*)))\chi_{\{n, n+1, ...\}}}_{X} \leqslant 1.
		\end{equation*}
		Moreover, it is not hard to see that
		\begin{equation*}
			C(f^*\chi_{\{n, n+1, ...\}}) \leqslant (C(f^*))\chi_{\{n, n+1, ...\}},
		\end{equation*}
		whence also
		\begin{equation*}
			\norm{F\left(C(f^*\chi_{\{n, n+1, ...\}})\right)}_{X} \leqslant \norm{F\left((C(f^*))\chi_{\{n, n+1, ...\}}\right)}_{X} \leqslant 1.
		\end{equation*}
		But this means that $\norm{f^*\chi_{\{n, n+1, ...\}}}_{C(X_F)} \leqslant 1$. Note also that $f^*\chi_{\{1, 2, ..., n-1\}} \in (X_F)_a$, so
		\begin{equation*}
			1 = d_{X_F}(f^*, (X_F)_a) \leqslant \norm{f - f^*\chi_{\{1, 2, ..., n-1\}}}_{X_F} = \norm{f^*\chi_{\{n, n+1, ...\}}}_{X_F} \leqslant 1.
		\end{equation*}
		Now, after some obvious modifications (such as converting integrals with sums or using the discrete Ces\`aro operator in place of his continuous
		counterpart) in the proof of equality (\ref{warunek rownosc}) from Corollary~\ref{corollary: kopia w funkcyjnym Calderonie}, we can show that
		\begin{equation*}
			\norm{f^*\chi_{\{n, n+1, ...\}}}_{C(X_F)} = 1.
		\end{equation*}
		Keeping in mind that the order continuous part of a rearrangement invariant sequence space is always non-trivial we can complete the proof
		by referring once again to Theorem~\ref{CX ma izometryczne l_infty}.
	\end{proof}

	The assumption that appeared in Theorem~\ref{CX ma izometryczne l_infty} may at first glance look quite wishfull, but it follows from the proof
	of the next result that rather natural subclass of the class of rearrangement invariant spaces satisfies this condition.
	
	\begin{proposition}
		{\it Let $X$ be a rearrangement invariant space on $(0,\infty)$ with the Fatou property such that $L_\infty \hookrightarrow X$
			and the Ces\`aro operator is bounded on $X$. Suppose also that the space $X$ satisfies the condition $({\mathscr D}_\infty)$
			from Corollary~\ref{corollary: de Jonge}. Then the space $CX$ contains a lattice isometric copy of $\ell_\infty$.}
	\end{proposition}
	\begin{proof}
		To begin with, take $f = \chi_{(0,\infty)}/\varphi_X(\infty)$. The assumption $L_\infty \hookrightarrow X$ implies that $\norm{f}_X = 1$.
		Moreover,
		\begin{equation*}
			\distE{f,X_a} = \lim_{n\rightarrow\infty}\norm{f^*\chi_{(n,\infty)}}_X = \lim_{n\rightarrow\infty}\norm{(f^*\chi_{(n,\infty)})^*}_X = \norm{f}_X = 1,
		\end{equation*}
		where the first equality follows from Corollary~\ref{corollary: formula na odleglosc} and the fact due to the condition $({\mathscr D}_\infty)$
		that every function supported on the set of finite measure is order continuous. Consequently, we have the equalities $\norm{f}_X = \distE{f,X_a} = 1$.
		Again, due to the condition $({\mathscr D}_\infty)$, the ideal $X_a$ is non-trivial, so we can use Theorem \ref{Twierdzenie Hudzika}
		to conclude that the space $X$ contains a lattice isometric copy of $\ell_\infty$. On the other hand, for such $f$ it is not hard to see that
		\begin{equation*}
			\norm{f^*\chi_{(b,\infty)}}_{CX} = 1 \text{ for every } b > 0.
		\end{equation*}
		In fact,
		\begin{align*}
		\norm{f^*\chi_{(b,\infty)}}_{CX} = \frac{1}{\varphi_X(\infty)} \norm{C(\chi_{(b,\infty)})}_X
										 = \frac{1}{\varphi_X(\infty)} \norm{(C(\chi_{(b,\infty)}))^*}_X
										 = \frac{1}{\varphi_X(\infty)} \norm{\chi_{(0,\infty)}}_X = 1.
		\end{align*}
		But this means that the condition in Theorem \ref{CX ma izometryczne l_infty} is satisfied with $a = 0$ and $b > 0$.
		We therefore conclude that the space $CX$ contains a lattice isometric copy of $\ell_\infty$ and the proof is done.
	\end{proof}

	Some spaces, such as Orlicz spaces (and certain variants of Orlicz spaces like Orlicz--Lorentz spaces or Musielak--Orlicz spaces;
	cf.~\cite[Remarks on p.~526]{Hu98}), have this (from our perspective) pleasant property that if they contain a lattice isomorphic
	copy of $\ell_\infty$, then they actually contain a lattice isometric copy of $\ell_\infty$. This simple observation along with
	the results from \cite{KT17} leads us to the following
	
	\begin{theorem} \label{CX ma kopie to X ma}
		{\it Let $X$ be an order continuous rearrangement invariant space with the Fatou property and let $F$ be an Orlicz
			function. Suppose that the Ces\`aro operator $C$ is bounded on the Calder{\'o}n--Lozanovski{\u \i} space $X_F$.
			If the Ces\`aro space $C(X_F)$ contains a lattice isometric copy of $\ell_\infty$, then $X_F$ contains it also.}
	\end{theorem}
	\begin{proof}
		Assume that $X_F$ does not contain a lattice isometric copy of $\ell_\infty$. Then the space $X_F$ does not contain even
		an order isomorphic copy of $\ell_\infty$. However, in view of the Lozanovski{\u \i} theorem \cite{Lo69},
		this means that $X_F$ is order continuous. At this stage, applying \cite[Theorem~3]{KT17}, we obtain that the space $CX$
		is order continuous as well. Consequently, the space $CX$ cannot contain a lattice isometric copy of $\ell_\infty$ and
		we are done.
	\end{proof}

	As we already mentioned, if $X$ is a Banach function space with a trivial ideal $X_a$, then the characterization given in
	Theorem~\ref{Twierdzenie Hudzika} is no longer valid in general. What seems to be quite interesting, the case of Ces\`aro
	spaces $CX$ (at least, when $X$ is a rearrangement invariant space) is different, because a suppport of the ideal $(CX)_a$
	is equal to a support of the space $X$ even if the ideal $X_a$ is trivial (see \cite[Lemma 8]{KT17}). Therefore, we can
	still use Theorem~\ref{Twierdzenie Hudzika} and obtain the next
	
	\begin{theorem} \label{CX ma kopie linfty Xa = 0}
			{\it Let $X$ be a rearrangement invariant space with the Fatou property such that the ideal $X_a$ is trivial.
			Suppose that there exists a function $f > 0$ with
			\begin{itemize}
				\item[(i)]  $C(f)(0^+) = 1$,
				\item[(ii)] $\norm{f}_{CX} = \norm{\emph{id} \colon X \rightarrow L_\infty}$.
			\end{itemize}
			Then the Ces\`aro space $CX$ contains a lattice isometric copy of $\ell_\infty$.}
	\end{theorem}
	\begin{proof} At the beginning observe that $X \hookrightarrow L_\infty$, because $X_a = \{ 0 \}$ (cf. \cite[Theorem~B]{KT17}).
		Moreover, since we can always replace the norm $\norm{\cdot}_X$ by a constant multiple of itself, it follows that without
		loss of generality we can assume that $\norm{\text{id} \colon X \rightarrow L_\infty} = 1$.
		
		Let $f$ be a function whose existence is guaranteed by our assumptions, that is to say, $C(f)(0^+)=1$ and $\norm{f}_{CX}
		= \norm{\text{id} \colon X \rightarrow L_\infty} = 1$. We will show that
		\begin{equation} \tag{$\ast$}
			\distC{f, (CX)_a} = 1.
		\end{equation}
		Indeed, we have the following inequalities
		\begin{align*}
		\distC{f, (CX)_a} & = \inf \{ \norm{f - g}_{CX} \colon g \in (CX)_a
		\text{ and } 0 < g \leqslant f \} \\
		& = \inf \{ \norm{C(f) - C(g)}_{X} \colon g \in (CX)_a \text{ and } 0
		< g \leqslant f \} \\
		& \quad\quad\quad (\text{because } f - g \geqslant 0) \\
		& \geqslant \inf \{ \norm{C(f) - C(g)}_{L_\infty} \colon g \in (CX)_a
		\text{ and } 0 < g \leqslant f \} \\
		& \quad\quad\quad (\text{in view of the embedding } X
		\xhookrightarrow{1} L_\infty) \\
		& \geqslant \inf \{ \abs{C(f)(0^+) - C(g)(0^+)} \colon g \in (CX)_a
		\text{ and } 0 < g \leqslant f \} = 1,
		\end{align*}
		where the first equality is due to Lemma~\ref{lemma 1 distance} and the last equality is the consequence of the fact that
		$C(g)(0^+) = 0$, whenever $g \in (CX)_a$ and the ideal $X_a$ is trivial (see \cite[Lemma~14]{KT17}). In summary, we proved
		the claim $(\ast)$, because we have shown that
		\begin{equation*}
			1 \leqslant \distC{f, (CX)_a} \leqslant \norm{f}_{CX} \leqslant 1.
		\end{equation*}
		
		Now, to complete the proof, it is enough to note that $\supp((CX)_a) = \supp(X)$ (see \cite[Lemma~8]{KT17}) and use once
		again Hudzik's result from \cite{Hu98} (cf. Theorem~\ref{Twierdzenie Hudzika}).
	\end{proof}
	
	\begin{remark}
		Let us now make a few remarks regarding the above Theorem~\ref{CX ma kopie linfty Xa = 0}:
		
		(a) The assumption about the symmetry of the space $X$ cannot in general be omitted (because, for example, the space $L_\infty(w)$ with
		the weight $w$ defined as $w(x) = x$ (for $x > 0$) contains a lattice isometric copy of $\ell_\infty$, but $C(L_\infty(w)) \equiv L_1$
		and the space $L_1$ is even separable).
		
		(b) It may happen (consider, for example, the space $X = L_1 \cap L_\infty$ on $(0,\infty)$) that the space $X$ contains a lattice
		isometric copy of $\ell_\infty$ but the Ces\`aro space $CX$ is trivial.
		
		(c) If $X$ is a rearrangement invariant space on $(0,1)$ with the Fatou property such that the ideal $X_a$ is trivial, then the space $X$
		is just $L_\infty$ up to equivalence of norms (for the proof see, e.g. \cite[Theorem~B]{KT17}).
	\end{remark}
	
	\begin{corollary} \label{corollary: X cap L_infty}
		{\it Let $X$ be a rearrangement invariant space on $(0,\infty)$ with the Fatou property such that the Ces\`aro space $CX$ is non-trivial.
			Suppose that either the ideal $X_a$ is non-trivial and the Ces\`aro operator $C$ is bounded on $X$ or $X_a = X$. Then the Ces\`aro
			space $C(X \cap L_\infty)$ contains a lattice isometric copy of $\ell_\infty$.}
	\end{corollary}
	\begin{proof}
		This fact follows (more or less) directly from Theorem~\ref{CX ma kopie linfty Xa = 0}. We just need to check a few details. Moreover,
		we will give the proof only in the case when the ideal $X_a$ is non-trivial and the operator $C$ is bounded on $X$, because the remaining
		case is much easier.
		
		First of all, let us note that
		\begin{equation*}
			X \cap L_\infty \overset{1}{\hookrightarrow} L_\infty \text{ (so, also } (X \cap L_\infty)_a = \{0\}) \text{ and } C(X \cap L_\infty) \neq \{0\}.
		\end{equation*}
		Put $f_\varepsilon = \chi_{(0,\varepsilon)}$, where $\varepsilon > 0$. We claim that
		\begin{equation} \tag{$\star$}
			\text{there exists } \varepsilon_0 > 0 \text{ such that } \norm{f_{\varepsilon_0}}_{CX} \leqslant 1.
		\end{equation}
		To show $(\star)$, it is enough to observe that $\varphi_X(0^+) = 0$ (this is actually equivalent to the fact that the ideal $X_a$ is non-trivial,
		see \cite[Theorem~5.5~(a), pp.~67--68]{BS88}; cf. \cite[Theorem~B]{KT17}) and the Ces\`aro operator $C$ is bounded on $X$, hence
		\begin{equation*}
			\norm{f_{\varepsilon}}_{CX} = \norm{C(\chi_{(0,\varepsilon)})}_X \leqslant \norm{C}_{X \rightarrow X} \varphi_X(\varepsilon) \rightarrow 0
																														\text{ as } \varepsilon \rightarrow 0.
		\end{equation*}
		Consequently, we have the following equalities
		\begin{equation*}
			\norm{f_{\varepsilon_0}}_{C(X \cap L_\infty)} = \max\{\norm{f_{\varepsilon_0}}_{CX}, \norm{f_{\varepsilon_0}}_{Ces_\infty}\}
														  = \max\{\norm{f_{\varepsilon_0}}_{CX}, 1\} = 1.
		\end{equation*}
		This means that we have defined the function $f_{\varepsilon_0}$ with
		\begin{equation*}
			C(f_{\varepsilon_0})(0^+) = 1 \text{ and } \norm{f_{\varepsilon_0}}_{C(X \cap L_\infty)} = 1 = \norm{\text{id} \colon X \cap L_\infty \rightarrow L_\infty}.
		\end{equation*}
		Now, because the ideal $(X \cap L_\infty)_a$ is trivial, we can use Theorem~\ref{CX ma kopie linfty Xa = 0} and end the proof.
	\end{proof}

	\begin{remark}
		(a) Let $X = L_p + L_\infty$ on $(0,\infty)$, where $1 < p < \infty$. It is clear that in this case the ideal $X_a$ is non-trivial\footnote{Precisely,
			it can be shown that $(L_p + L_\infty)_a = \{f \in L_p + L_\infty \colon f^*(\infty) = 0\}$ (remembering that
			$\norm{f}_{L_p + L_\infty} \asymp (\int_0^1 f^*(t)^p dt)^{1/p}$ (see~\cite[p.~109]{BL76}) one can, for example, modify the proof of an analogous fact
			for the case when $p = 1$ included in \cite{KT17}; cf.~\cite[pp.~115--116]{KPS82}).}
			and the Ces\`aro operator is bounded on $X$. However, it is also clear\footnote{Just note that $L_\infty \overset{1}{\hookrightarrow} L_p + L_\infty$,
			so $\norm{f}_{L_\infty} \leqslant \max\{\norm{f}_{L_p + L_\infty}, \norm{f}_{L_\infty}\} \leqslant \norm{f}_{L_\infty}$.} that
			$X \cap L_\infty \equiv L_\infty$. Therefore, using Theorem~\ref{CX ma kopie linfty Xa = 0} (via Corollary~\ref{corollary: X cap L_infty})
			we re-prove Proposition~4.9 from \cite{KKM21} which states that the space $Ces_\infty$ contains, as one would expect, a lattice isometric
			copy of $\ell_\infty$.
			
		(b) Let $X = L_F$ be an Orlicz space generated by an Orlicz function $F$ such that $b_F \coloneqq \sup\{x > 0 \colon F(x) < \infty\} = 1$,
			$F(1) \leqslant 1$ and the left derivative of $F$ at $x = 1$ is finite. Then there exists an Orlicz function, say $G$, with $b_G = \infty$ such
			that $L_F \equiv L_G \cap L_\infty$. The proof of this fact is analogous to the proof of \cite[Theorem~12.1~(a), p.~99]{Ma89}. However, to
			get an isometry instead of isomorphism, we need to notice two things: (1) due to the assumption regarding the derivative of $F$, we can
			always extend $F$ (but not necessarily in a uniqe way!) to an Orlicz function, say $\widetilde{F}$, with $b_{\widetilde{F}} = \infty$;
			(2) the function $(\widetilde{F} \vee F_\infty)(x) \coloneqq \max\{\widetilde{F}(x), F_\infty(x)\}$, where
			\begin{align*}
			F_\infty(x) =
			\begin{cases}
			0 & \text{for } 0 \leqslant x \leqslant 1\\
			\infty & \text{for } x > 1
			\end{cases},
			\end{align*}
			is equal to $F$.  Now, the embedding $L_F \overset{1}{\hookrightarrow} L_{\widetilde{F}} \cap L_\infty$ follows directly from \cite[Theorem~12.1]{Ma89}.
			Therefore, it remains to show that also
			\begin{equation} \tag{$\sharp$}
				L_{\widetilde{F}} \cap L_\infty \overset{1}{\hookrightarrow} L_F.
			\end{equation}
			Take $f \in L_{\widetilde{F}} \cap L_\infty$ with $\max\{\norm{f}_{L_{\widetilde{F}}}, \norm{f}_{L_\infty}\} = 1$. Let us consider two cases:
			
			$1^{\circ}$. $\norm{f}_{L_\infty} = 1$. Then $\abs{f(x)} \leqslant 1 = b_F$, so $\int_0^\infty F(\abs{f(x)})dx = \int_0^\infty \widetilde{F}(\abs{f(x)})dx \leqslant 1$.
			Moreover, since $\norm{f}_{L_\infty} = 1 = b_F$, it follow that $\int_0^\infty F(\abs{f(x)}/\lambda)dx = \infty$ for every $\lambda < 1$.
			Therefore, $\norm{f}_{L_F} = 1$ as well.
			
			$2^{\circ}$. $\norm{f}_{L_{\widetilde{F}}} = 1$ and $\norm{f}_{L_\infty} < 1$. Then
			$\int_0^\infty F(\abs{f(x)})dx = \int_0^\infty \widetilde{F}(\abs{f(x)})dx \leqslant 1$, so $\norm{f}_{L_F} \leqslant 1$. In order to obtain a contradiction,
			suppose that $\norm{f}_{L_F} < 1$. Then $\int_0^\infty F(\abs{f(x)}/\lambda_0)dx \leqslant 1$ for some $\lambda_0 < 1$. Hence, $\abs{f}/\lambda_0 \leqslant 1 = b_F$
			and, consequently, also $\int_0^\infty \widetilde{F}(\abs{f(x)}/\lambda_0)dx \leqslant 1$. However, this implies that $\norm{f}_{L_{\widetilde{F}}} < 1$,
			which is impossible. The proof of $(\sharp)$ has been completed.
			
			In particular, if $1 \leqslant p < \infty$ and the Orlicz function $F_{p,\infty}$ is given in the following way
			\begin{align*}
			F_{p,\infty}(x) =
			\begin{cases}
				x^p & \text{for } 0 \leqslant x \leqslant 1\\
			\infty & \text{for } x > 1
			\end{cases},
			\end{align*}
			then $L_{F_{p,\infty}} \equiv L_p \cap L_\infty$.
			
			The above spaces provided some natural examples of spaces which satisfy the assumptions of Corollary~\ref{corollary: X cap L_infty} (of course,
			in order to use Corollary~\ref{corollary: X cap L_infty}, the functions $F$ and $G$ cannot be completely arbitrary; however, translating the
			assumptions of Corollary~\ref{corollary: X cap L_infty} into the language of the Orlicz spaces - the $\Delta_2$-condition and the Matuszewska--Orlicz
			indices (cf.~\cite[Proposition~2.b.5, p.~139]{LT79}) will surely appear - does not present much difficulty).
			
		(c) Let $\varphi$ be an increasing concave function with $\varphi(0^+) = 1$ (this normalization is basically inessential). Since the norm in the Lorentz
			space $\Lambda_\varphi$ is given by the formula
			\begin{equation*}
				\norm{f}_{\Lambda_\varphi} = \norm{f}_{L_\infty} + \int_0^\infty f^*(t)\varphi'(t)dt
										   = \norm{\left(\norm{f}_{\Lambda_{\int_0^t \varphi'(s)ds}}, \norm{f}_{L_\infty}\right)}_{\ell_1^2},
			\end{equation*}
			it follows that
			\begin{equation*}
				\Lambda_\varphi \equiv \Lambda_{\int_0^t \varphi'(s)ds} \cap L_\infty \overset{1}{\hookrightarrow} L_\infty \text{ and }
										\left(\Lambda_{\int_0^t \varphi'(s)ds}\right)_a \neq \{0\},
			\end{equation*}
			where, of course, the intersection space is considered with the equivalent two-dimensional $\ell_1$-norm given above (after all, all norms
			on $\mathbb{R}^2$ are equivalent). However, in this case the proof of Corollary~\ref{corollary: X cap L_infty} does not work. On the other hand,
			if we consider the general family of equivalent norms on the space $X \cap L_\infty$ of the form
			\begin{equation*}
				\norm{f}_{X \oplus_F L_\infty} = \norm{(\norm{f}_X, \norm{f}_{L_\infty})}_{F} \text{ for } f \in X \cap L_\infty,
			\end{equation*}
			where $\norm{\cdot}_F$ is a norm on $\mathbb{R}^2$ with the ideal property, then for our proof to work, it is enough to require that
			\begin{equation*}
				\norm{(x_0,1)}_{F} = \norm{(0,1)}_F = 1 \text{ for some } x_0 > 0,
			\end{equation*}
			which means that the unit sphere in $(\mathbb{R}^2, \norm{\cdot}_F)$ contains an order interval beginning at point $(0,1)$. For example,
			the \enquote{hexagonal} norm
			\begin{equation*}
			\norm{(x,y)}_{hex} = \max\left\{\abs{x + \frac{\sqrt{3}}{3}y}, \abs{x - \frac{\sqrt{3}}{3}y},\frac{2\sqrt{3}}{3}\abs{y}\right\} \text{ for } x,y \in \mathbb{R},
			\end{equation*}
			seems to be a good candidate for such a norm. But we will stop here.
		
		(d) The Ces\`aro operator $C$ is not bounded on the Zygmund space $L\log L$ but the Ces\`aro space $C(L\log L)$ is non-trivial (cf.~\cite{KKM21})
			and the space $L\log L$ is order continuous. On the other hand, the Ces\`aro operator is bounded on $L_\infty$ but the ideal $(L_\infty)_a$ is trivial.
	\end{remark}
	
	\section{Appendix. Local approach to order continuity in abstract Ces\`aro sequence spaces}
	
	We are going to present here a result characterizing the ideal of order continuous elements in the Ces\`aro sequence spaces. This result,
	as well as it's proof, is analogous to Theorem~7 from \cite{KT17} (to be fair, due to the fact that $X_a$ is always non-trivial subspace
	of $X$, whenever $X$ is a Banach sequence space, it is much simpler than its function counterpart). However, in order to relieve the reader
	from the tedious obligation to check all the details of this proof on his own, we rather prefer to provide a brief sketch of the argument.
	
	\begin{proposition}
		{\it Let $X$ be a rearrangement invariant sequence space such that the Ces\`aro operator $C$ is bounded on $X$. Then $(CX)_a = C(X_a)$.}
	\end{proposition}
	\begin{proof}
		To prove the equality $(CX)_a = C(X_a)$ we will show the following inclusions
		\begin{equation*}
		C(X_a) \subset (CX)_a \subset (CX)_b \subset C(X_a).
		\end{equation*}
		The first inclusion $C(X_a) \subset (CX)_a$ holds true even for Banach sequence spaces and its proof is {\it mutatis mutandis} the same
		as Lemma 11 in \cite{KT17}, while the second inclusion $(CX)_a \subset (CX)_b$ follows directly from \cite[Theorem~3.11, p.~18]{BS88}.
		Consequently, it remains only to show that $(CX)_b \subset C(X_a)$. We will do this in two steps.
		
		$1^{\circ}$. Take $\chi_{A}$, where $A \subset \mathbb{N}$ and $\max(A) \coloneqq n_0 < \infty$. 
		Let $n>n_0$. Since $p_X > 1$\footnote{Symbol $p_X$ denote here the lower Boyd index. We will only remaind that the classical Hardy operator $C$
		is bounded in a rearrangement invariant space $X$ if and only if $p_X > 1$ (for details we refer, for example, to
		\cite[Theorem~6.6, p.~138]{KPS82}, \cite[pp.~129--131]{LT79}) or \cite[pp.~126--129]{KMP07} and references given there).},
		so $\ell_p \hookrightarrow X$ for all $1 < p < p_X$ (this fact follows from Proposition~2.b.3 in \cite{LT79},
		which is admittedly formulated for rearrangement invariant function spaces, but exactly as noted at the top of the page 132, after some
		notation changes due to the use of discrete versions of Boyd's indices, his proof remains essentially unchanged in the comparison to the
		proof of the afromentioned Proposition~2.b.3) and we have the following inequalities
		\begin{equation*}
		\norm{C(\chi_A)\chi_{\{n,n+1,\dots\}}}_{X} \leqslant \norm{\frac{1}{k}\chi_{\{n,n+1,\dots\}}(k)}_{X}
												\preccurlyeq \norm{\frac{1}{k}\chi_{\{n,n+1,\dots\}}(k)}_{\ell_p} \rightarrow 0,
		\end{equation*}
		as $n \rightarrow \infty$.
		
		$2^{\circ}$. Take $a = \{a_n\}_{n=1}^\infty \in (CX)_b$. Let $\{s_n\}_{n=1}^\infty$ be a sequence of sequences such that $s_n$
		converges to $a$ in the norm of $CX$ and $\#(\supp(s_n)) < \infty$ for all $n \in \mathbb{N}$. It follows from the previous step
		that $\{s_n\}_{n=1}^\infty \subset C(X_a)$. Now, using the reverse triangle inequality, we get
		\begin{equation*}
		\norm{C(\abs{a}) - C(\abs{s_n})}_X \leqslant \norm{C(\abs{a - s_n})}_X = \norm{a - s_n}_{CX} \rightarrow 0,
		\end{equation*}
		as $n \rightarrow \infty$. Since $X_a$ is a closed ideal of $X$, so $C(\abs{a}) \in X_a$ and the proof is completed.
	\end{proof}

	The immediate consequence of the above result is the following fact (which, in the case of the abstract Ces\`aro function spaces,
	was the main result of \cite{KT17}).

	\begin{corollary}
		{\it Let $X$ be a rearrangement invariant sequence space such that the Ces\`aro operator $C$ is bounded on $X$. Then $CX$
			is order continuous if and only if $X$ is order continuous as well.}
	\end{corollary}

\end{document}